\documentclass{article}
\usepackage[totalwidth=480pt,totalheight=680pt]{geometry}

\usepackage{amssymb}
\usepackage{amsmath}
\usepackage{amsthm}
\usepackage{physics}
\usepackage{enumitem}
\usepackage{hyperref}
\hypersetup{colorlinks=true}

\newcommand{\bR}{\mathbb{R}}
\newcommand{\cmpl}{\backslash}
\def\closure#1{{\overline{#1}}}
\def\inner#1#2{{\left\langle#1,#2\right\rangle}}
\def\Laplace{{\triangle}}
\newcommand{\swcorner}{\mathbin{\vrule height 1.6ex depth 0pt width
0.13ex\vrule height 0.13ex depth 0pt width 1.3ex}}

\newcommand{\ti}{{\to\infty}}
\newcommand{\tz}{{\to0}}
\def\r#1{{\mathrm{#1}}}
\def\ceil#1{{\left\lceil#1\right\rceil}}
\def\simto{{\xrightarrow{\sim}}}
\def\verynewline{{\begin{flushleft}\end{flushleft}}}
\parskip=\smallskipamount
\theoremstyle{plain}
\newtheorem {THM}{Theorem}
\numberwithin{THM}{section}
\newtheorem {PROP}[THM]{Proposition}
\newtheorem {CORO}[THM]{Corollary}
\newtheorem {LEM}[THM]{Lemma}

\theoremstyle{definition}

\newtheorem {DEF}[THM]{Definition}
\newtheorem*{CONV}{Convention}
\title{Willmore Flow of Complete Surfaces}
\author{Long-Sin Li}
\date{December 16, 2022}
\begin{document}
\maketitle
\fontsize{12}{18pt}\selectfont

\begin{abstract}
    We consider the Willmore flow equation for complete, properly immersed surfaces in $\bR^n$.
    Given bounded geometry on the initial surface, we extend the result in \cite{Kuwert2002GradientFF} and prove short time existence and uniqueness of the Willmore flow.
    We also show that a complete Willmore surface with low Willmore energy must be a plane, and that a Willmore flow with low initial energy and Euclidean volume growth must converge smoothly to a plane.
\end{abstract}


\section{Introduction}
The Willmore energy of an immersed surface $f:\Sigma^2\looparrowright\bR^n$ is defined as
\begin{align*}
    \mathcal{W}(f):=\frac12\int_\Sigma|A|^2\dd{\mu},
\end{align*}
where $A$ denotes the second fundamental form.
By Gauss-Bonnet theorem, if $f$ is a closed surface,
\begin{align*}
    \mathcal{W}(f)=\int_\Sigma|A^0|^2\dd{\mu}+2\pi\chi(\Sigma)=\frac12\int_\Sigma|H|^2\dd{\mu}-2\pi\chi(\Sigma),
\end{align*}
where $A^0$ denotes the trace-free part of $A$.
Therefore, the two aforementioned expressions are used for the definition in some literature without causing an essential difference.
The first variation of $\mathcal{W}$ is given by the Willmore tensor:
\begin{align*}
    \mathbf{W}(f)=\Laplace H+Q(A^0)H,
\end{align*}
where $Q$ is defined by
\begin{align*}
    Q(\eta)\phi=g^{ik}g^{j\ell}\eta_{ij}\inner{\eta_{k\ell}}{\phi}_{N_\Sigma}
\end{align*}
for tensors $\eta\in\Gamma(N_\Sigma\otimes\r{Sym}^2(T^\ast\Sigma))$ and $\phi\in\Gamma(N_\Sigma)$.
Note that the first variation formula holds whether $\Sigma$ is compact or non-compact.
Therefore, given initial data $f_0:\Sigma\looparrowright\bR^n$, we consider the Willmore flow equation
\begin{align}\begin{cases}
    \partial_tf=-\mathbf{W}(f),\\
    f\big|_{t=0}=f_0.
\end{cases}\label{pdeOriginal}
\end{align}

There are open questions regarding singularities of Willmore flows on compact surfaces, including existence of finite-time singularities and classification of singularity types.
In \cite[Theorem 1.2]{Kuwert2002GradientFF}, the authors showed that finite-time singularities require energy concentration.
In \cite[Theorem 1.1]{10.1215/00127094-2009-014}, the authors showed that blowups are not compact.
In \cite[Theorem 1.4]{McWh}, the authors showed an upper bound for the existence time of locally constrained Willmore flows, while the upper bound increases to infinity as the PDE converges to the classical Willmore flow equation.
It could hence be interesting to approximate Willmore blowups by complete surfaces.
Thus we would like to extend the study on Willmore flow of closed surfaces to complete, properly immersed surfaces in $\bR^n$.

For closed surfaces in $\bR^n$, Kuwert and Sch\"atzle proved \cite[Theorem 1.2]{Kuwert2002GradientFF}, a short-time existence result for smooth solutions with a lower bound for the lifespan, depending only on concentration of curvature on the initial surface.
Our main result is a generalization of their theorem for complete surfaces:

\noindent{\bf Theorem \ref{existence3}. }{\it
Let $f_0:\Sigma\to\bR^n$ be a smooth, complete, properly immersed surface in $\bR^n$.
Then there exist $\varepsilon_1>0$ and $c_1>0$, both depending only on $n$, such that whenever the initial energy concentration condition
\begin{align*}
    \int_{\Sigma\cap B_\varrho(x)}|A|^2\dd{\mu}\leq\varepsilon_1\phantom{=}\text{when }t=0\text{, }\forall x\in\bR^n
\end{align*}
holds for some $\varrho>0$, the Willmore flow equation (\ref{pdeOriginal}) has a smooth solution $f:\Sigma\times[0,T)\to\bR^n$.
Moreover, the maximal existence time $T$ is at least $c_1^{-1}\varrho^4$.
}

To prove our theorem, we consider solutions to a weighted PDE
\begin{align}\begin{cases}
    \partial_tf=-\theta^r\mathbf{W}(f),\\
    f\big|_{t=0}=f_0,
\end{cases}\label{pdeCuttoff}
\end{align}
where $0\leq\theta\leq1$ is a smooth function defined on the ambient space and $r$ is a sufficiently big number.
First, we recover that the a priori estimates in \cite{Kuwert2002GradientFF} hold for the weighted PDE.
From the estimates, we can show that for a Willmore flow that is defined for $t\in[0,T)$, $f$ converges as $t\to T$ and hence can be extended to $[0,T]$, provided $T<c_1^{-1}\varrho^4$.
However, well-posedness of (\ref{pdeCuttoff}), which allows $f$ to extended over $T$, generally only holds when $\Sigma$ is compact and $\theta>0$.
To solve (\ref{pdeOriginal}), we approximate solutions to (\ref{pdeOriginal}) with solutions to (\ref{pdeCuttoff}) with compactly supported $\theta$;
next, we can view $\Sigma\cap[\theta>0]$ as a subset of another surface which is compact, and hence the aforementioned solution solves (\ref{pdeCuttoff}) on a compact surface; and
finally, we can approximate $\theta$ on the compact surface with different choices for $\theta$ that are positive everywhere on the compact surface.

For compact surfaces, the vector field $\mathbf{W}(f)$ is the gradient of the energy $\mathcal{W}(f)$ and hence we have energy identity for any family of surfaces:
\begin{align*}
    \dv{t}\mathcal{W}(f)=\int_\Sigma\inner{\mathbf{W}(f)}{\partial_tf}\dd{\mu}.
\end{align*}
In particular, energy decreases along a negative gradient flow.
For complete surfaces, energy may escape to infinity and hence decrease even faster:

\noindent{\bf Corollary \ref{energyIneq}. }{\it
If $\mathcal{W}(f_0)<\infty$ and $f$ is the Willmore flow constructed in Theorem \ref{existence3}, then we have
\begin{align*}
    \int_\Sigma|A|^2\dd{\mu}+\int_0^t\int_\Sigma|\mathbf{W}(f)|^2\dd{\mu}\dd{t'}\leq\int_\Sigma|A|^2\dd{\mu}\bigg|_{t=0}.
\end{align*}
}

Using Sobolev inequalities, initial non-concentration conditions for $A,\ldots,\nabla^5A$ implies uniform bounds for $A,\ldots,\nabla^3A$, which give us sufficient flatness to obtain the following uniqueness result for the fourth-order PDE.

\noindent{\bf Theorem \ref{uniquenessCoro}. }{\it
Assume that $f_0:\Sigma\to\bR^n$ is a smooth, complete, properly immersed surface in $\bR^n$ such that
\begin{align*}
    \liminf_{R\ti}R^{-4}\mu_0\big(B_R(0)\big)=0\text{, and}
\end{align*}
for some $\varrho>0$ and $M>0$,
\begin{align*}\begin{cases}
    \displaystyle\int_{\Sigma_0\cap B_\varrho(x)}|A|^2\dd{\mu_0}\leq\varepsilon_1,&\forall x\in\bR^n\\
    \displaystyle\int_{\Sigma_0\cap B_\varrho(x)}|\nabla^kA|^2\dd{\mu_0}\leq M,&\forall x\in\bR^n\text{ and }k=1,\ldots,5.
\end{cases}\end{align*}
Let $f=f_1$ and $f=f_2$ be two solutions to the Willmore flow equation (\ref{pdeOriginal}), then there exists $0<\widehat{T}\leq T$ such that $f_1=f_2$ for all $0\leq t<\widehat{T}$.
}

Lastly, the a priori estimates also imply gap phenomena for small energy $\mathcal{W}(f)$.
The result is also explained by more general theorems such as \cite[Theorem 2.7]{10.4310/jdg/1090348128} and \cite[Theorem 1, (2)]{Wh15}.

\noindent{\bf Theorem \ref{gap2}. }{\it
If $f_0:\Sigma\to\bR^n$ is a complete, smooth, properly immersed Willmore surface with $\mathcal{W}(f_0)\leq\frac12\varepsilon_0$, then $\Sigma$ is a plane, where $\varepsilon_0>0$ only depends on $n$.
}

As a result, we can prove that Willmore flows with small initial energy converge to planes:

\noindent{\bf Corollary \ref{gap3}. }{\it
Let $f:\Sigma\times[0,\infty)\to\bR^n$ be a solution to (\ref{pdeOriginal}).
Assume that $\mathcal{W}(f_0)\leq\frac12\varepsilon_0$ and that $\mu\big(B_R(0)\big)\leq c(R)$ for all $t\in[0,\infty)$ and $R>0$.
Then as $t\ti$, any subsequence has a further subsequence such that $\Sigma$ converges locally smoothly, up to diffeomorphisms, to a plane $L:\bR^2\to\bR^n$ in the sense as in Definition \ref{defCvg}.
}

In the statement, we assume space-time bounds to guarantee convergence and to avoid having a sum of planes as the limit.
Alternatively, we have the same conclusion if we assume an Euclidean area growth rate for the initial surface:

\noindent{\bf Corollary \ref{gap4}. }{\it
Let $f:\Sigma\times[0,\infty)\to\bR^n$ be a solution to (\ref{pdeOriginal}).
Assume that $\mathcal{W}(f_0)\leq\frac12\varepsilon_0$ and that
\begin{align*}
    \liminf_{R\ti}R^{-2}\mu_0\big(B_R(0)\big)<\infty.
\end{align*}
Then as $t\ti$, any subsequence has a further subsequence such that $\Sigma$ converges to a plane $L:\bR^2\to\bR^n$ in the sense as in Definition \ref{defCvg}.
}

{\bf Acknowledgements.} We thank Prof. Jeffrey Streets for his insight and comments, which are very helpful throughout this project.

\section{Conventions}\label{conventions}
First, we list some notations that are used throughout the article.
\begin{itemize}
\item $f:\Sigma^2\times[0,T)\to\bR^n$ is smooth, and $\nabla$ is the Levi-Civita connection of $(\Sigma,(f(-,t))^\ast g_{\bR^n})$.\\ $Df=f_\ast$ is the derivative along $\Sigma$.\\
$\Sigma_0$ denotes the surface $f(\Sigma\times\{0\})$.\\
Unless specified, $\Sigma$ is identified with the surface $f(\Sigma\times\{t\})$ for arbitrarily fixed $t\in[0,T)$.
\item $\mu$: the measure on $\Sigma$ induced by $f$.\\
$\mu_0$: $\mu$ at $t=0$.
\item $\phi$: A tensor of class $\Gamma((T^\ast\Sigma)^{\otimes r_\phi}\otimes N_\Sigma)$, where $r_\phi$ is a non-negative integer and $N_\Sigma$ is the normal bundle on $\Sigma$.
\item $\Laplace=-\nabla^\ast\nabla$, where $\nabla^\ast$ is the formal adjoint of $\nabla$.
\item $s,r\geq2$: sufficiently large positive integers.
\item $\displaystyle P_k^m=\sum_{i_1+\cdots+i_r=m}\nabla^{i_1}A\ast\cdots\ast\nabla^{i_k}A$ with coefficients being bounded by some $c(n,s,r)$.
\end{itemize}
Next, we find a smooth function $\chi$ defined on $\bR$ such that
    \begin{align*}\begin{cases}
        \chi\text{ is decreasing,}\\
        \chi(x)=1\text{ for all }x\leq0\text{, and}\\
        \chi(x)=0\text{ for all }x\geq1.
    \end{cases}\end{align*}
Moreover, we fix this choice so that $\sup|D^k\chi|$ only depends on $k$.
Next, we require $\widehat\gamma$ and $\widehat\theta$ to be functions defined on $\bR^n$ such that for some given $K>0$,
\begin{align}\begin{cases}
    \widehat\gamma\text{ and }\widehat\theta\text{ are smooth,}\\
    0\leq\widehat\gamma,\widehat\theta\leq1\text{ while also both are not identically }0,\\
    \widehat\gamma\widehat\theta\text{ has compact support, and}\\
    \text{For all }k\geq1,\,|D^k\widehat\gamma|\leq K^k\sup|D^k\chi|\text{ and }|D^k\widehat\theta|\leq K^k\sup|D^k\chi|.
\end{cases}\label{cutoff}\end{align}
\begin{LEM}\label{cutoffConstruct}
    Given any $x_1,x_2\in\bR^n$, $R_1,R_2>0$, and $0<K_1,K_2\leq K$, we can let
    \begin{align*}
        \widehat\gamma(x)=\chi\left(\frac{|x-x_1|-R_1}{K_1}\right)
        \phantom{=}
        \text{ and }
        \phantom{=}
        \widehat\theta(x)=\chi\left(\frac{|x-x_2|-R_2}{K_2}\right)
    \end{align*}
    so that they satisfy (\ref{cutoff}).
\end{LEM}
Let $\gamma=\widehat\gamma\big|_\Sigma$ and $\theta=\widehat\theta\big|_\Sigma$.
We can estimate the derivatives of $\gamma^s\theta^r$:
\begin{LEM}
    For all $k\geq1$,
    \begin{align*}
        |\nabla^k(\gamma^s\theta^r)|
        &\leq c\,\bigg(\gamma^{\max(s-k,0)}\theta^{\max(r-k,0)}K^k
        +\sum_{\substack{k_1+k_2+k_3=k\\ k_1,k_2\geq1\\ k_3\geq0}}
        \gamma^{\max(s-k_1,0)}\theta^{\max(s-k_1,0)}K^{k_1}\big|P_{k_2}^{k_3}\big|\bigg),
    \end{align*}
    where $c=c(s,r,k)$.
    The cases when $k=1,2$ are especially frequently used:
    \begin{align*}
        |\nabla(\gamma^s\theta^r)|
        \leq c\,\gamma^{s-1}\theta^{r-1}K,
    \end{align*}
    and
    \begin{align*}
        |\nabla^2(\gamma^s\theta^r)|
        &\leq c\,\big(\gamma^{s-2}\theta^{r-2}K^2+\gamma^{s-1}\theta^{r-1}K|A|\big).
    \end{align*}
\end{LEM}
\begin{proof}
    The proof is clear by induction, while we only show the cases $k=1,2$.
    Let $(u,v)$ be a normal coordinate at $p\in\Sigma$ and $e_1=\partial_u$, $e_2=\partial_v$.
    We have
    \begin{align*}
        &\nabla(\gamma^s\theta^r)(e_i)=D\big(\widehat\gamma^s\widehat\theta^r\big)(e_i),
    \end{align*}
    and
    \begin{align*} 
        &\nabla^2(\gamma^s\theta^r)(e_i,e_j)=D^2\big(\widehat\gamma^s\widehat\theta^r\big)(e_i,e_j)+D\big(\widehat\gamma^s\widehat\theta^r\big)\big(A(e_i,e_j)\big).
    \end{align*}
\end{proof}
\section{General inequalities}
In this section, we derive several different inequalities regarding $L^p$ norms of tensors with the cutoff functions, including interpolation inequalities and multiplicative Sobolev inequalities.
These inequalities are later applied in the context of Willmore flows.

\subsection{Interpolation inequalities}

\begin{LEM}[Local convexity implies global convexity]\label{convex1}
Let $\{a_m\}_{m=0}^M$ be a sequence of non-negative real numbers, and $c_1(\varepsilon),\ldots,c_{M-1}(\varepsilon)$ be a sequence of functions taking non-negative values such that for all $\varepsilon>0$ and $1\leq m\leq(M-1)$,
\begin{align*}
    a_m\leq\varepsilon a_{m+1}+c_m(\varepsilon)\,a_{m-1}.
\end{align*}
Then for all $1\leq m_0\leq(M-1)$ we also have
\begin{align*}
    a_{m_0}\leq\varepsilon a_M+c\,a_0,
\end{align*}
where $c=c(\varepsilon,M,c_1,\ldots,c_{M-1})$.
\end{LEM}
\begin{PROP}\label{interpolation1}
Let $0<m_0<M$ be integers, $2\leq j<\infty$, and $p,q\geq j$.
If $s\geq Mp$ and $r\geq Mq$, then for all $\varepsilon>0$,
\begin{align*}
    &K^{M-m_0}\left(\int_\Sigma\gamma^{s-(M-m_0)p}\theta^{r-(M-m_0)q}|\nabla^{m_0}\phi|^j\dd{\mu}\right)^{1/j}\\
    &\leq\varepsilon\left(\int_\Sigma\gamma^{s}\theta^{r}|\nabla^M\phi|^j\dd{\mu}\right)^{1/j}+c\,K^M\left(\int_\Sigma\gamma^{s-Mp}\theta^{r-Mq}|\phi|^j\dd{\mu}\right)^{1/j}
\end{align*}
where $c=c(s,r,\varepsilon,r_\phi,M,j)$.
\end{PROP}
\begin{proof}
Throughout this proof, we let $c=c(s,r,r_\phi,M,j)$.
Define
\begin{align*}
    a_m=K^{M-m}\displaystyle\left(\int_\Sigma\gamma^{s-(M-m)p}\theta^{r-(M-m)q}|\nabla^m\phi|^j\dd{\mu}\right)^{1/j}.
\end{align*}
Then for each $m_0+1\leq m\leq M-1$, using integration by parts and H\"older's inequality,
\begin{align*}
    a_m^j
    &\leq c\,K^{(M-m)j}\int_\Sigma\Big(\gamma^{s-(M-m)p}\theta^{r-(M-m)q}|\nabla^{m+1}\phi|\\
    &\phantom{\leq K^{(M-m)j}\int\int}+\big|\nabla\big(\gamma^{s-(M-m)p}\theta^{r-(M-m)q}\big)\big|\,|\nabla^m\phi|\Big)|\nabla^m\phi|^{j-2}\,|\nabla^{m-1}\phi|\dd{\mu}\\
    &\leq c\,K^{(M-m)j}\int_\Sigma\gamma^{s-(M-m)p}\theta^{r-(M-m)q}|\nabla^{m+1}\phi|\,|\nabla^m\phi|^{j-2}\,|\nabla^{m-1}\phi|\dd{\mu}\\
    &\phantom{==} +c\,K^{(M-m)j+1}\int_\Sigma\gamma^{s-(M-m)p-1}\theta^{r-(M-m)q-1}|\nabla^m\phi|^{j-1}\,|\nabla^{m-1}\phi|\dd{\mu}\\
    &\leq c\,(a_{m+1}+a_m)a_m^{j-2}a_{m-1}.
\end{align*}
Thus for arbitrary $\varepsilon>0$,
\begin{align*}
    &a_m
    \leq c\,\sqrt{(a_{m+1}+a_m)a_{m-1}}
    \leq\frac\varepsilon2a_{m+1}+\frac12a_m+c\,(1+\varepsilon^{-1})a_{m-1},
\end{align*}
which implies
\begin{align*}
    a_m\leq\varepsilon a_{m+1}+c\,(1+\varepsilon^{-1})a_{m-1}.
\end{align*}
By Lemma \ref{convex1}, we can conclude the result.
\end{proof}
\begin{PROP}\label{interpolation3}
Let $M\geq2$ be an integer, $\alpha\geq0$, $2\leq j<\infty$, and $p,q\geq0$.
If $s\geq\max(2p,Mj)$ and $r\geq\max(2q,Mj)$, then for all $\varepsilon>0$,
\begin{align*}
    &\alpha\left(\int_\Sigma\gamma^{s-p}\theta^{r-q}|\nabla^{M-1}\phi|^j\dd{\mu}\right)^{1/j}\\
    &\leq\varepsilon\left[\left(\int_\Sigma\gamma^s\theta^r|\nabla^M\phi|^j\dd{\mu}\right)^{1/j}+c\,K^M\left(\int_\Sigma\gamma^{s-Mj}\theta^{r-Mj}|\phi|^j\dd{\mu}\right)^{1/j}\right]\\
    &\phantom{==}+c\,\alpha^2\varepsilon^{-1}\left(\int_\Sigma\gamma^{s-2p}\theta^{r-2q}|\nabla^{M-2}\phi|^j\dd{\mu}\right)^{1/j},
\end{align*}
where $c=c(s,r,r_\phi,M)$.
\end{PROP}
\begin{proof}
Using integration by parts,
\begin{align*}
    &\alpha^j\int_\Sigma\gamma^{s-p}\theta^{r-q}|\nabla^{M-1}\phi|^j\dd{\mu}\\
    &\leq c\,\alpha^j\int_\Sigma\big(\gamma^{s-p}\theta^{r-q}|\nabla^M\phi|+\gamma^{s-p-1}\theta^{r-q-1}K|\nabla^{M-1}\phi|\big)|\nabla^{M-1}\phi|^{j-2}|\nabla^{M-2}\phi|\dd{\mu}\\
    &\leq c\,\alpha^j\left(\int_\Sigma\gamma^s\theta^r|\nabla^M\phi|^j\dd{\mu}+\int_\Sigma\gamma^{s-j}\theta^{r-j}K^j|\nabla^{M-1}\phi|^j\dd{\mu}\right)^{1/j}\\
    &\phantom{==}\cdot\left(\int_\Sigma\gamma^{s-p}\theta^{r-q}|\nabla^{M-1}\phi|^j\dd{\mu}\right)^{(j-2)/j}\left(\int_\Sigma\gamma^{s-2p}\theta^{s-2q}|\nabla^{M-2}\phi|^j\dd{\mu}\right)^{1/j}\\
    &\leq\frac{\varepsilon^j}4\int_\Sigma\gamma^s\theta^r|\nabla^M\phi|^j\dd{\mu}+\frac{\varepsilon^j}2\int_\Sigma\gamma^{s-j}\theta^{r-j}K^j|\nabla^{M-1}\phi|^j\dd{\mu}\\
    &\phantom{==}+\frac{\alpha^j}2\int_\Sigma\gamma^{s-p}\theta^{r-q}|\nabla^{M-1}\phi|^j\dd{\mu}+c\,\alpha^{2j}\varepsilon^{-j}\int_\Sigma\gamma^{s-2p}\theta^{s-2q}|\nabla^{M-2}\phi|^j\dd{\mu},
\end{align*}
and hence
\begin{align*}
    &\alpha^j\int_\Sigma\gamma^{s-p}\theta^{r-q}|\nabla^{M-1}\phi|^j\dd{\mu}\\
    &\leq\frac{\varepsilon^j}2\int_\Sigma\gamma^s\theta^r|\nabla^M\phi|^j\dd{\mu}+\varepsilon^j\int_\Sigma\gamma^{s-j}\theta^{r-j}K^j|\nabla^{M-1}\phi|^j\dd{\mu}
    .+c\,\alpha^{2j}\varepsilon^{-j}\int_\Sigma\gamma^{s-2p}\theta^{s-2q}|\nabla^{M-2}\phi|^j\dd{\mu}.
\end{align*}
Next, by Proposition \ref{interpolation1},
\begin{align*}
    &\alpha^j\int_\Sigma\gamma^{s-p}\theta^{r-q}|\nabla^{M-1}\phi|^j\dd{\mu}\\
    &\leq\varepsilon^j\int_\Sigma\gamma^s\theta^r|\nabla^M\phi|^j\dd{\mu}+c\,\alpha^{2j}\varepsilon^{-j}\int_\Sigma\gamma^{s-2p}\theta^{s-2q}|\nabla^{M-2}\phi|^j\dd{\mu}+c\,\varepsilon^jK^{Mj}\int_\Sigma\gamma^{s-Mj}\theta^{r-Mj}|\phi|^j\dd{\mu},
\end{align*}
which is equivalent to the inequality to be proved.
\end{proof}
\begin{PROP}\label{interpolation1.1}
Let $0\leq m_1<m_0<M$ be integers, $\alpha\geq0$, $2\leq j<\infty$, and $p,q\geq0$.
If $s\geq\max\big(Mp-m_1(p-j),Mj\big)$ and $r\geq\max\big(Mq-m_1(q-j),Mj\big)$, then for all $\varepsilon>0$,
\begin{align*}
    &\alpha^{M-m_0}\left(\int_\Sigma\gamma^{s-(M-m_0)p}\theta^{r-(M-m_0)q}|\nabla^{m_0}\phi|^j\dd{\mu}\right)^{1/j}\\
    &\leq\varepsilon\left(\int_\Sigma\gamma^s\theta^r|\nabla^M\phi|^j\dd{\mu}\right)^{1/j}+c\,\alpha^{M-m_1}\left(\int_\Sigma\gamma^{s-(M-m_1)p}\theta^{r-(M-m_1)q}|\nabla^{m_1}\phi|^j\dd{\mu}\right)^{1/j}\\
    &\phantom{==}+c\,\big(K^M+\alpha^{M-m_1}K^{m_1}\big)\left(\int_{[\gamma\theta>0]}|\phi|^j\dd{\mu}\right)^{1/j},
\end{align*}
where $c=c(s,r,\varepsilon,r_\phi,M)$.
\end{PROP}
\begin{proof}
By Proposition \ref{interpolation3}, for all $m=(m_1+1),\ldots,(M-1)$, there exists $b_m\geq0$ such that
\begin{align*}
    &\alpha\left(\int_\Sigma\gamma^{s-(M-m)p}\theta^{r-(M-m)q}|\nabla^m\phi|^j\dd{\mu}\right)^{1/j}\\
    &\leq\varepsilon\,\bigg[\left(\int_\Sigma\gamma^{s-(M-m-1)p}\theta^{r-(M-m-1)q}|\nabla^{m+1}\phi|^j\dd{\mu}\right)^{1/j}\\
    &\phantom{=\varepsilon\,\alpha^{M-m-1}\bigg[}+b_{m+1}\,K^{m+1}\left(\int_\Sigma\gamma^{s-Mp+(m+1)(p-j)}\theta^{r-Mq+(m+1)(q-j)}|\phi|^j\dd{\mu}\right)^{1/j}\bigg]\\
    &\phantom{==}+\alpha^2\varepsilon^{-1}b_{m-1}\left(\int_\Sigma\gamma^{s-(M-m+1)p}\theta^{s-(M-m+1)q}|\nabla^{m-1}\phi|^j\dd{\mu}\right)^{1/j}.
\end{align*}
For convenience, let $b_M=0$.
We can find a sequence $\{\widehat{b}_m\}$ that satisfies
\begin{align*}\begin{cases}
    \widehat{b}_m>b_m & \text{for all }m=m_1,\ldots,M\text{, and}\\
    \widehat{b}_m\leq\sqrt{\widehat{b}_{m-1}(\widehat{b}_{m+1}-b_{m+1})} & \text{for all }m=m_1+1,\ldots,M-1,
\end{cases}\end{align*}
so that
\begin{align*}
    &\alpha\,\widehat{b}_mK^m\left(\int_\Sigma\gamma^{s-Mp+m(p-j)}\theta^{r-Mq+m(q-j)}|\phi|^j\dd{\mu}\right)^{1/j}\\
    &\leq\alpha\,\widehat{b}_mK^m\left(\int_\Sigma\gamma^{s-Mp+(m+1)(p-j)}\theta^{r-Mq+(m+1)(q-j)}|\phi|^j\dd{\mu}\right)^{1/(2j)}\\
    &\phantom{==\alpha\,\widehat{b}_mK^m}\cdot\left(\int_\Sigma\gamma^{s-Mp+(m-1)(p-j)}\theta^{r-Mq+(m-1)(q-j)}|\phi|^j\dd{\mu}\right)^{1/(2j)}\\
    &\leq\varepsilon(\widehat{b}_{m+1}-b_{m+1})K^{m+1}\left(\int_\Sigma\gamma^{s-Mp+(m+1)(p-j)}\theta^{r-Mq+(m+1)(q-j)}|\phi|^j\dd{\mu}\right)^{1/j}\\
    &\phantom{==}+\alpha^2\varepsilon^{-1}\widehat{b}_{m-1}K^{m-1}\left(\int_\Sigma\gamma^{s-Mp+(m-1)(p-j)}\theta^{r-Mq+(m-1)(q-j)}|\phi|^j\dd{\mu}\right)^{1/j}.
\end{align*}
Let
\begin{align*}
    a_m
    &=\alpha^{M-m}\bigg[\left(\int_\Sigma\gamma^{s-(M-m)p}\theta^{r-(M-m)q}|\nabla^m\phi|^j\dd{\mu}\right)^{1/j}\\
    &\phantom{=\alpha^{M-m}\bigg[}+\widehat{b}_mK^m\left(\int_\Sigma\gamma^{s-Mp+m(p-j)}\theta^{r-Mq+m(q-j)}|\phi|^j\dd{\mu}\right)^{1/j}\bigg]
\end{align*}
so that by the two inequalities above, we get
\begin{align*}
    a_m\leq\varepsilon a_{m+1}+\varepsilon^{-1}a_{m-1}.
\end{align*}
Therefore, by Lemma \ref{convex1}, there exists some $c=c(s,r,\varepsilon,r_\phi,M)$ such that
\begin{align*}
    &\alpha^{M-m_0}\left(\int_\Sigma\gamma^{s-(M-m_0)p}\theta^{r-(M-m_0)q}|\nabla^{m_0}\phi|^j\dd{\mu}\right)^{1/j}\\
    &\leq a_{m_0}\\
    &\leq\varepsilon a_M+c\,a_{m_1}\\
    &\leq\varepsilon\left(\int_\Sigma\gamma^s\theta^r|\nabla^M\phi|^j\dd{\mu}\right)^{1/j}+c\,\alpha^{M-m_1}\left(\int_\Sigma\gamma^{s-(M-m_1)p}\theta^{r-(M-m_1)q}|\nabla^{m_1}\phi|^j\dd{\mu}\right)^{1/j}\\
    &\phantom{==}+c\,\big(K^M+\alpha^{M-m_1}K^{m_1}\big)\left(\int_\Sigma\gamma^{s-\max\big(Mp-m_1(p-j),Mj\big)}\theta^{r-\max\big(Mq-m_1(q-j),Mj\big)}|\phi|^j\dd{\mu}\right)^{1/j}.
\end{align*}
\end{proof}

\subsection{Sobolev inequalities}
\begin{THM}[Michael-Simon Sobolev inequality \cite{https://doi.org/10.1002/cpa.3160260305}]\label{sobolevMS}
Let $f:\Sigma^2\to\bR^n$ be a smooth immersion. Then for any $u\in C^1_c(\Sigma)$ we have
\begin{align*}
\left(\int_\Sigma u^2\dd{\mu}\right)^{1/2}\leq c_n\left(\int_\Sigma|\nabla u|\dd{\mu}+\int_\Sigma|H|\,|u|\dd{\mu}\right),
\end{align*}
where $c_n$ is a constant only depending on $n$.
\end{THM}
\begin{CONV}
Let $h\in C^1_c(\Sigma)$ satisfy $\|\nabla h\|_\infty\leq c\,K$, where $c=c(n,s,r)$.
For example, $h=\gamma^s\theta^r$.
\end{CONV}
We can rewrite \cite[Lemma 5.1]{Kuwert2002GradientFF} as the following.
\begin{LEM}\label{sobolev51}
Let $1\leq p,q,w\leq\infty$ satisfy $\frac1p+\frac1q=\frac1w$ and $\alpha,\beta\in\bR$ satisfy $\alpha+\beta=1$. For any $b\geq\max(\alpha q,\beta p)$ and $-\frac1p\leq t\leq\frac1q$, we have
\begin{align*}
    \|h^{b/(2w)}\nabla\phi\|_{2w}^2
    &\leq c\,\|h^{b(p^{-1}+t)}\nabla^2\phi\|_p\,\|h^{b(q^{-1}-t)}\phi\|_q+c\,bK\|h^{b/p-\beta}\nabla\phi\|_p\,\|h^{b/q-\alpha}\phi\|_q,
\end{align*}
where $c=c(n,w,r_\phi)$.
\end{LEM}
\begin{LEM}\label{sobolev2coro}
For all $u\in C^1(\Sigma)$ and $ h\in C^1(\Sigma)$ such that their product has compact support, $0<m<\infty$, and $2<p<\infty$,
\begin{align*}
    \|h^\alpha u\|_\infty\leq c\,\|u\|_m^{1-\alpha}\big(\|h\nabla u\|_p+K\|u\|_p+\|h uH\|_p\big)^\alpha,
\end{align*}
where
\begin{align*}
    \alpha=\frac{2p}{(p-2)m+2p}
\end{align*}
and $c=c(n,m,p)$.
\end{LEM}
\begin{proof}
Let
\begin{align*}\begin{cases}
    \displaystyle q=\frac1{1-\frac1p}\in(1,2),\\
    \tau_0=m/q\in(0,\infty),\\
    \beta_0=0,\vspace{1.5mm}\\
    \tau_\nu=\left(\frac2q\right)^\nu\cdot\tau_0+\frac{\left(\frac2q\right)^{\nu+1}-\frac2q}{\frac2q-1}\in(0,\infty)&\text{for all }\nu=1,2,\ldots\text{, and}\vspace{3mm}\\
    \beta_\nu=\frac{\left(\frac2q\right)^{\nu+1}-\frac2q}{\frac2q-1}\in(0,\infty)&\text{for all }\nu=1,2,\ldots;
\end{cases}\end{align*}
so that the numbers solve the inductive formulas (where $\nu=0,1,\ldots$)
\begin{align*}\begin{cases}
    \tau_{\nu+1}q=2(1+\tau_\nu),\vspace{2mm}\\
    \displaystyle\frac{\beta_{\nu+1}}{\tau_{\nu+1}}=\frac{1+\beta_\nu}{1+\tau_\nu},
\end{cases}
\end{align*}
namely,
\begin{align*}
    \frac{\beta_{\nu+1}}{1+\beta_\nu}=\frac{\tau_{\nu+1}}{1+\tau_\nu}=\frac2q.
\end{align*}
As a result, we see that
\begin{align*}
    &\big\|h^{\beta_{\nu+1}/\tau_{\nu+1}}u\big\|_{\tau_{\nu+1}q}^{1+\tau_\nu}\\
    &=\big\|h^{(1+\beta_\nu)/(1+\tau_\nu)}u\big\|_{2(1+\tau_\nu)}^{1+\tau_\nu}\\
    &=\big\|h^{1+\beta_\nu}u^{1+\tau_\nu}\big\|_2\\
    &\leq c_n\,\Big(\big\|\nabla\big( h^{1+\beta_\nu}u^{1+\tau_\nu}\big)\big\|_1+\big\|h^{1+\beta_\nu}u^{1+\tau_\nu}H\big\|_1\Big)\tag{Theorem \ref{sobolevMS}}\\
    &\leq c_n\,\Big((1+\tau_\nu)\big\|h^{1+\beta_\nu}u^{\tau_\nu}\nabla u\big\|_1+(1+\beta_\nu)\big\|h^{\beta_\nu}u^{1+\tau_\nu}\nabla h\big\|_1+\big\|h^{1+\beta_\nu}u^{1+\tau_\nu}H\big\|_1\Big)\\
    &\leq c_n\,\big\|h^{\beta_\nu}u^{\tau_\nu}\big\|_q\Big((1+\tau_\nu)\big\|h\nabla u\big\|_p+(1+\beta_\nu)\big\|u\nabla h\big\|_p+\big\|h uH\big\|_p\Big)\\
    &\leq AB_\nu\big\|h^{\beta_\nu/\tau_\nu}u\big\|_{\tau_\nu q}^{\tau_\nu},
\end{align*}
where
\begin{align*}
    A=c_n\,\Big(\big\|h\nabla u\big\|_p+\big\|u\nabla h\big\|_p+\big\|h uH\big\|_p\Big)
    \phantom{=}\text{and}\phantom{=}
    B_\nu=1+\tau_\nu\geq1+\beta_\nu.
\end{align*}
Observe that
\begin{align*}
    \left(\frac{q}2\right)^\nu B_\nu\leq c_\ast:=\tau_0+\frac{2}{2-q}.
\end{align*}
Therefore,
\begin{align*}
    \big\|h^{\beta_{\nu+1}/\tau_{\nu+1}}u\big\|_{\tau_{\nu+1}q}
    \leq\left(Ac_\ast\left(\frac2q\right)^\nu\right)^{1/(1+\tau_\nu)}\big\|h^{\beta_\nu/\tau_\nu}u\big\|_{\tau_\nu q}^{\tau_\nu/(1+\tau_\nu)}.
\end{align*}
Define $\varepsilon_\nu=\tau_\nu/(1+\tau_\nu)$ so that we get from the previous inequality that
\begin{align*}
    \big\|h^{\beta_\nu/\tau_\nu}u\big\|_{\tau_\nu q}
    \leq\big\|u\big\|_{\tau_0q}^{\varepsilon_0\times\cdots\times\varepsilon_{\nu-1}}\prod_{j=0}^{\nu-1}\left(Ac_\ast\left(\frac2q\right)^j\right)^{\varepsilon_j\times\cdots\times\varepsilon_{\nu-1}/\tau_j}
\end{align*}
On the left hand side, we have
\begin{align*}
    &\lim_{\nu\ti}\tau_\nu q=\infty\text{ and }\lim_{\nu\ti}\frac{\beta_\nu}{\tau_\nu}=\frac{2}{(2-q)\tau_0+2}=\alpha
    \implies\lim_{\nu\ti}\big\|h^{\beta_\nu/\tau_\nu}u\big\|_{\tau_\nu q}=\big\|h^\alpha u\big\|_\infty
\end{align*}
On the right hand side, observe that
\begin{align*}
    \varepsilon_j\times\cdots\times\varepsilon_{\nu-1}
    =\frac{\tau_j}{1+\tau_{\nu-1}}\left(\frac2q\right)^{\nu-j-1}
\end{align*}
so that we have
\begin{align*}
    \lim_{\nu\ti}\big(\varepsilon_0\times\cdots\times\varepsilon_{\nu-1}\big)
    =\frac{(2-q)\tau_0}{(2-q)\tau_0+2}=1-\alpha
    \implies\lim_{\nu\ti}\big\|h^{\beta_0/\tau_0}u\big\|_{\tau_0q}^{\varepsilon_0\times\cdots\varepsilon_{\nu-1}}\leq\big\|u\big\|_m^{1-\alpha},
\end{align*}
\begin{align*}
    \lim_{\nu\ti}\sum_{j=0}^{\nu-1}\big(\varepsilon_j\times\cdots\times\varepsilon_{\nu-1}/\tau_j\big)=\frac{2}{(2-q)\tau_0+2}=\alpha\implies\lim_{\nu\ti}\prod_{j=0}^{\nu-1}(Ac_\ast)^{\varepsilon_j\times\cdots\times\varepsilon_{\nu-1}/\tau_j}=(Ac_\ast)^\alpha,
\end{align*}
and
\begin{align*}
    \lim_{\nu\ti}\sum_{j=0}^{\nu-1}\big(j\times\varepsilon_j\times\cdots\times\varepsilon_{\nu-1}/\tau_j\big)=\frac{4}{(2-q)^2\tau_0+2(2-q)}.
\end{align*}
In summary,
\begin{align*}
    \big\|h^\alpha u\big\|_\infty
    \leq c\,\big\|u\big\|_m\Big(\big\|h\nabla u\big\|_p+\big\|u\nabla h\big\|_p+\big\|huH\big\|_p\Big)
    \leq c\,\big\|u\big\|_m\Big(\big\|h\nabla u\big\|_p+K\big\|u\big\|_p+\big\|huH\big\|_p\Big),
\end{align*}
where
\begin{align*}
    c=(c_nc_\ast)^\alpha\cdot\left(\frac2q\right)^{4/[(2-q)^2\tau_0+2(2-q)]}
\end{align*}
only depends on $n$, $m$, and $p$.
\end{proof}
\begin{LEM}\label{sobolev1}
    For all $u\in C^1_c(\Sigma)$ and $2<p<\infty$,
    \begin{align*}
        &\|u\|_p\leq c\,\|u\|_2^{2/p}\big(\|\nabla u\|_2+\|Hu\|_2\big)^{1-2/p},
    \end{align*}
    where $c=c(n,p)$.
\end{LEM}
\begin{proof}
    For all positive integer $\tau$, by Theorem \ref{sobolevMS}, we have
    \begin{align*}
        \|u^{1+\tau}\|_2
        &\leq c_n\,(1+\tau)\left(\int_\Sigma|u|^\tau|\nabla u|\dd{\mu}+\int_\Sigma|H|\,|u|^{1+\tau}\dd{\mu}\right)\\
        &\leq c_n\,(1+\tau)\|u^\tau\|_2\big(\|\nabla u\|_2+\|Hu\|_2\big).
    \end{align*}
    As a result, by induction,
    \begin{align*}
        &\|u^\tau\|_2\leq c_n^{\tau-1}\,(\tau!)\|u\|_2\big(\|\nabla u\|_2+\|Hu\|_2\big)^{\tau-1},
    \end{align*}
    or equivalently,
    \begin{align*}
        \|u\|_{2\tau}
        &\leq c_n\,\Big[\tau!\,\|u\|_2\big(\|\nabla u\|_2+\|Hu\|_2\big)^{\tau-1}\Big]^{1/\tau}.
    \end{align*}
    Finally, take $\tau=\ceil{\frac p2}$ so that 
    \begin{align*}
        \|u\|_p
        &\leq\|u\|_2^{\frac{2\tau-p}{(\tau-1)p}}\|u\|_{2\tau}^{\frac{\tau(p-2)}{(\tau-1)p}}\\
        &\leq c_n(\tau!)^{1/\tau}\|u\|_2^{\frac{2\tau-p}{(\tau-1)p}}\|u\|_2^{\frac{p-2}{(\tau-1)p}}\big(\|\nabla u\|_2+\|Hu\|_2\big)^{\frac{p-2}p}\\
        &=c_n(\tau!)^{1/\tau}\|u\|_2^{\frac2p}\big(\|\nabla u\|_2+\|Hu\|_2\big)^{1-\frac2p}.
    \end{align*}
\end{proof}

\subsection{Geometric inequalities}
For convenience, we rewrite \cite[Lemma 4.2]{Kuwert2002GradientFF}, replacing $\gamma^4$ with $\gamma^s\theta^r$:
\begin{LEM}\label{sobolev42coro}
If $s,r\geq4$, then
\begin{align*}
    &\int_\Sigma\gamma^s\theta^r\big(|\nabla A|^2|A|^2+|A|^6\big)\dd{\mu}\\
    &\leq c\int_{[\gamma\theta>0]}|A|^2\dd{\mu}\int_\Sigma\gamma^s\theta^r\big(|\nabla^2A|^2+|A|^6\big)\dd{\mu}+c\,K^4\left(\int_{[\gamma\theta>0]}|A|^2\dd{\mu}\right)^2,
\end{align*}
where $c=c(n,s,r)$.
Moreover, there exists $\varepsilon_0>0$, only depending on $n$, $s$, and $r$, such that whenever
\begin{align}
    \int_{[\gamma\theta>0]}|A|^2\dd{\mu}\leq\varepsilon_0,\label{energyNonconcentrating}
\end{align}
we have
\begin{align*}
    \int_\Sigma\gamma^s\theta^r\big(|\nabla A|^2|A|^2+|A|^6\big)\dd{\mu}
    \leq\int_\Sigma\gamma^s\theta^r|\nabla^2A|^2+c\,K^4\left(\int_{[\gamma\theta>0]}|A|^2\dd{\mu}\right)^2.
\end{align*}
\end{LEM}
\begin{LEM}\label{sobolev42higher}
If $s\geq6$ and $r\geq8$, then we can choose $\varepsilon_0$ so that assuming (\ref{energyNonconcentrating}), we have
\begin{align*}
    \int_\Sigma\gamma^s\theta^{r-2}K^2|A|^8\dd{\mu}
    \leq\int_\Sigma\gamma^s\theta^r|\nabla A|^4|A|^2+K^8\int_{[\gamma\theta>0]}|A|^2\dd{\mu}.
\end{align*}
\end{LEM}
\begin{proof}
By Theorem \ref{sobolevMS},
\begin{align*}
    &\int_\Sigma\gamma^s\theta^{r-2}K^2|A|^8\dd{\mu}\\
    &\leq c\left(\int_\Sigma\gamma^{s/2}\theta^{r/2-1}K|\nabla A|\,|A|^3\dd{\mu}+\int_\Sigma\gamma^{s/2-1}\theta^{r/2-2}K^2|A|^4\dd{\mu}+\int_\Sigma\gamma^{s/2}\theta^{r/2-1}K|A|^5\dd{\mu}\right)^2\\
    &\leq c\left(\int_\Sigma\gamma^{s/2}\theta^{r/2}|\nabla A|^2|A|^2\dd{\mu}+\int_\Sigma\gamma^{s/2-1}\theta^{r/2-2}K^2|A|^4\dd{\mu}+\int_\Sigma\gamma^{s/2}\theta^{r/2-1}K|A|^5\dd{\mu}\right)^2\\
    &\leq c\left(\int_\Sigma\gamma^{s/2}\theta^{r/2}|\nabla A|^2|A|^2\dd{\mu}+\int_\Sigma\gamma^{s/2-3}\theta^{r/2-4}K^4|A|^2\dd{\mu}+\int_\Sigma\gamma^{s/2}\theta^{r/2-1}K|A|^5\dd{\mu}\right)^2\\
    &\leq c\,\varepsilon_0\left(\int_\Sigma\gamma^s\theta^r|\nabla A|^4|A|^2\dd{\mu}+\int_\Sigma\gamma^s\theta^{r-2}K^2|A|^8\dd{\mu}+K^8\int_{[\gamma\theta>0]}|A|^2\dd{\mu}\right).
\end{align*}
We require $c\,\varepsilon_0\leq\frac12$ to obtain the claimed statement.
\end{proof}
\begin{PROP}\label{interpolation2}
If $s\geq2$ and $r\geq4$, then
\begin{align*}
    \int_\Sigma\gamma^s\theta^r|\nabla A|^4\dd{\mu}
    \leq c\,\|\theta^{r/4}A\|_{\infty,[\gamma>0]}^2\left(\int_\Sigma\gamma^s\theta^{r/2}|\nabla^2A|^2\dd{\mu}+K^2\int_\Sigma\gamma^{s-2}\theta^{r/2-2}|\nabla A|^2\dd{\mu}\right),
\end{align*}
where $c=c(n,s,r)$.
\end{PROP}
\begin{proof}
Using integration by parts,
\begin{align*}
    \int_\Sigma\gamma^s\theta^r|\nabla A|^4\dd{\mu}
    &\leq c\int_\Sigma\big(\gamma^s\theta^r|\nabla^2A|\,|\nabla A|^2+\gamma^{s-1}\theta^{r-1}K|\nabla A|^3\big)|A|\dd{\mu}\\
    &\leq c\,\|\theta^{r/4}A\|_{\infty,[\gamma>0]}\left(\int_\Sigma\gamma^s\theta^r|\nabla A|^4\dd{\mu}\right)^{1/2}\\
    &\phantom{==}\cdot\left[\left(\int_\Sigma\gamma^s\theta^{r/2}|\nabla^2A|^2\dd{\mu}\right)^{1/2}+\left(\int_\Sigma\gamma^{s-2}\theta^{r/2-2}K^2|\nabla A|^2\dd{\mu}\right)^{1/2}\right]\\
    &\leq\frac12\int_\Sigma\gamma^s\theta^r|\nabla A|^4\dd{\mu}\\
    &\phantom{==}+c\,\|\theta^{r/4}A\|_{\infty,[\gamma>0]}^2\left(\int_\Sigma\gamma^s\theta^{r/2}|\nabla^2A|^2\dd{\mu}+K^2\int_\Sigma\gamma^{s-2}\theta^{r/2-2}|\nabla A|^2\dd{\mu}\right),
\end{align*}
and hence we can obtain the stated inequality.
\end{proof}
\begin{PROP}\label{interpolation2higher}
If $s\geq6$ and $r\geq20$, then we can choose $\varepsilon_0$ so that assuming (\ref{energyNonconcentrating}), we have
\begin{align*}
    &\int_\Sigma\gamma^s\theta^r|\nabla A|^4|A|^2\dd{\mu}\\
    &\leq c\,K^2\int_\Sigma\gamma^{s-2}\theta^{r-2}|\nabla A|^3\dd{\mu}+\|\theta^{r/4}A\|_{\infty,[\gamma>0]}^4\int_{[\theta>0]}\gamma^s|\nabla^2A|^2\dd{\mu}\\
    &\phantom{==}+c\,\big(K^5\|\theta^{r/4}A\|_{\infty,[\gamma>0]}^3+K^8\big)\int_{[\gamma\theta>0]}|A|^2\dd{\mu},
\end{align*}
where $c=c(n,s,r)$.
\end{PROP}
\begin{proof}
First,
\begin{align*}
    \int_\Sigma\gamma^s\theta^r|\nabla A|^4|A|^2\dd{\mu}\leq\|\theta^{r/4}A\|_{\infty,[\gamma>0]}^2\int_\Sigma\gamma^s\theta^{r/2}|\nabla A|^4\dd{\mu}.
\end{align*}
Next, using integration by parts,
\begin{align*}
    &\int_\Sigma\gamma^s\theta^{r/2}|\nabla A|^4\dd{\mu}\\
    &\leq c\int_\Sigma\big(\gamma^s\theta^{r/2}|\nabla^2A|\,|\nabla A|^2+\gamma^{s-1}\theta^{r/2-1}K|\nabla A|^3\big)|A|\dd{\mu}\\
    &\leq c\,\|\theta^{r/4}A\|_{\infty,[\gamma>0]}\left(\int_\Sigma\gamma^s\theta^{r/2}|\nabla A|^4\dd{\mu}\right)^{1/2}\left(\int_{[\theta>0]}\gamma^s|\nabla^2A|^2\dd{\mu}\right)^{1/2}\\
    &\phantom{==}+c\,K\left(\int_\Sigma\gamma^s\theta^{r/2}|\nabla A|^4\dd{\mu}\right)^{3/4}\left(\int_\Sigma\gamma^{s-4}\theta^{r/2-4}|A|^4\dd{\mu}\right)^{1/4}\\
    &\leq\frac12\int_\Sigma\gamma^s\theta^{r/2}|\nabla A|^4\dd{\mu}+c\,\|\theta^{r/4}A\|_{\infty,[\gamma>0]}^2\int_{[\theta>0]}\gamma^s|\nabla^2A|^2\dd{\mu}+c\,K^4\int_\Sigma\gamma^{s-4}\theta^{r/2-4}|A|^4\dd{\mu},
\end{align*}
so that we have
\begin{align*}
    &\int_\Sigma\gamma^s\theta^r|\nabla A|^4|A|^2\dd{\mu}\\
    &\leq c\,\|\theta^{r/4}A\|_{\infty,[\gamma>0]}^4\int_{[\theta>0]}\gamma^s|\nabla^2A|^2\dd{\mu}+c\,K^4\|\theta^{r/4}A\|_{\infty,[\gamma>0]}^2\int_\Sigma\gamma^{s-4}\theta^{r/2-4}|A|^4\dd{\mu}.
\end{align*}
Next, by Theorem \ref{sobolevMS},
\begin{align*}
    &\int_\Sigma\gamma^{s-4}\theta^{r/2-4}|A|^4\dd{\mu}\\
    &\leq c\left(\int_\Sigma\gamma^{s/2-2}\theta^{r/4-2}|\nabla A|\,|A|\dd{\mu}+K\int_\Sigma\gamma^{s/2-3}\theta^{r/4-3}|A|^2\dd{\mu}+\int_\Sigma\gamma^{s/2-2}\theta^{r/4-2}|A|^3\dd{\mu}\right)^2\\
    &\leq c\,\varepsilon_0\left(\int_\Sigma\gamma^{s-4}\theta^{r/2-4}|\nabla A|^2\dd{\mu}+\int_\Sigma\gamma^{s-4}\theta^{r/2-4}|A|^4\dd{\mu}\right)+c\,\varepsilon_0K^2\int_{[\gamma\theta>0]}|A|^2\dd{\mu},
\end{align*}
and hence we can require $c\,\varepsilon_0\leq\frac12$ so that we have
\begin{align*}
    &\int_\Sigma\gamma^s\theta^r|\nabla A|^4|A|^2\dd{\mu}\\
    &\leq c\,\|\theta^{r/4}A\|_{\infty,[\gamma>0]}^4\int_{[\theta>0]}\gamma^s|\nabla^2A|^2\dd{\mu}+c\,K^4\|\theta^{r/4}A\|_{\infty,[\gamma>0]}^2\int_\Sigma\gamma^{s-4}\theta^{r/2-4}|\nabla A|^2\dd{\mu}\\
    &\phantom{==}+c\,K^6\|\theta^{r/4}A\|_{\infty,[\gamma>0]}^2\int_{[\gamma\theta>0]}|A|^2\dd{\mu}.
\end{align*}
Next, by Proposition \ref{interpolation1.1} with $\alpha=K^{1/2}\|\theta^{r/4}A\|_{\infty,[\gamma>0]}^{1/2}$, we have
\begin{align*}
    &K^2\|\theta^{r/4}A\|_{\infty,[\gamma>0]}^2\int_\Sigma\gamma^{s-4}\theta^{r/2-4}|\nabla A|^2\dd{\mu}\\
    &\leq K^2\|\theta^{r/4}A\|_{\infty,[\gamma>0]}^2\int_\Sigma\gamma^{s-4}\theta^{(r-2)/3}|\nabla A|^2\dd{\mu}\tag{$r\geq20$}\\
    &\leq\int_\Sigma\gamma^{s-2}\theta^{r-2}|\nabla^3A|^2\dd{\mu}+c\,\big(K^3\|\theta^{r/4}A\|_{\infty,[\gamma>0]}^3+K^6\big)\int_{[\gamma\theta>0]}|A|^2\dd{\mu}.
\end{align*}
In summary,
\begin{align*}
    &\int_\Sigma\gamma^s\theta^r|\nabla A|^4|A|^2\dd{\mu}\\
    &\leq c\,K^2\int_\Sigma\gamma^{s-2}\theta^{r-2}|\nabla A|^3\dd{\mu}+\|\theta^{r/4}A\|_{\infty,[\gamma>0]}^4\int_{[\theta>0]}\gamma^s|\nabla^2A|^2\dd{\mu}\\
    &\phantom{==}+c\,\big(K^5\|\theta^{r/4}A\|_{\infty,[\gamma>0]}^3+K^8\big)\int_{[\gamma\theta>0]}|A|^2\dd{\mu}.
\end{align*}
\end{proof}


\begin{LEM}\label{sobolev43}
{\normalfont (\cite[Lemma 4.3]{Kuwert2002GradientFF})}
\begin{enumerate}[label=(\roman*)]
    \item We have
    \begin{align*}
        &\|\phi\|_{\infty,[\gamma=1]}^4
        \leq c\,\|\phi\|_{2,[\gamma>0]}^2\big(\|\nabla^2\phi\|_{2,[\gamma>0]}^2+\|\phi\|_{2,[\gamma>0]}^2+\|\,|A|^4|\phi|^2\|_{1,[\gamma>0]}\big),
    \end{align*}
    where $c=c(n,r_\phi,K)$.
    \item Moreover, assuming (\ref{energyNonconcentrating}), we have
\begin{align*}
    &\|A\|_{\infty,[\gamma=1]}^4\leq c\,\|A\|_{2,[\gamma>0]}^2\big(\|\nabla^2A\|_{2,[\gamma>0]}^2+\|A\|_{2,[\gamma>0]}^2\big).
\end{align*}
\end{enumerate}
\end{LEM}
The following corollary refines both the previous lemma and \cite[Lemma 2.8]{10.4310/jdg/1090348128}.
\begin{CORO}\label{sobolev3}
If $r\geq6$, assuming (\ref{energyNonconcentrating}), we have
\begin{align*}
    \|\theta^{r/4}A\|_{\infty,[\gamma=1]}^4\leq c\,\|A\|_{2,[\gamma\theta>0]}^2\big(\|\theta^{r/2}\nabla^2A\|_{2,[\gamma>0]}^2+\|A\|_{2,[\gamma\theta>0]}^2\big)
\end{align*}
where $c=c(n,r,K)$.
\end{CORO}
\begin{proof}
First, by Lemma \ref{sobolev2coro} with $m=2$, $p=4$, etc.,
\begin{align*}
    \|\gamma^2\theta^{r/4}A\|_\infty
    &\leq c\,\|A\|_{2,[\gamma\theta>0]}^{1/3}\big(\|\gamma^3\theta^{3r/8}\nabla A\|_4+\|\gamma^2\theta^{3r/8-1}A\|_4+\|\gamma^3\theta^{3r/8}|A|^2\|_4\big)^{2/3}.
\end{align*}
Next, by Lemma \ref{sobolev51} with $\phi=A$, $p=2$, etc.,
\begin{align*}
    \|\gamma^3\theta^{3r/8}\nabla A\|_4^2
    \leq c\,\big(\|\gamma^4\theta^{r/2}\nabla^2A\|_2\|\gamma^2\theta^{r/4}A\|_\infty+\|\gamma^3\theta^{r/2-1}\nabla A\|_2\|\gamma^2\theta^{r/4}A\|_\infty\big).
\end{align*}
Moreover, we have
\begin{align*}
    \|\gamma^3\theta^{r/2-1}\nabla A\|_2&\leq c\,\big(\|\gamma^4\theta^{r/2}\nabla^2A\|_2+\|A\|_{2,[\gamma\theta>0]}\big), \tag{Lemma \ref{interpolation1}}\\
    \|\gamma^2\theta^{3r/8-1}A\|_4^4&\leq\|\gamma^2\theta^{r/4}A\|_\infty^2\|A\|_{2,[\gamma\theta>0]}^2,\\
    \|\gamma^3\theta^{3r/8}|A|^2\|_4^4&\leq\|\gamma^2\theta^{r/4}A\|_\infty^2\|\gamma^8\theta^r|A|^6\|_1\text{, and}\\
    \|\gamma^8\theta^r|A|^6\|_1&\leq c\,\big(\|\gamma^4\theta^{r/2}\nabla^2A\|_2^2+\|A\|_{2,[\gamma\theta>0]}^2\big). \tag{Lemma \ref{sobolev42coro}}
\end{align*}
Combining all the inequalities above,
\begin{align*}
    \|\gamma^2\theta^{r/4}A\|_\infty\leq c\,\|A\|_{2,[\gamma\theta>0]}^{1/3}\|\gamma^2\theta^{r/4}A\|_\infty^{1/3}\big(\|\gamma^4\theta^{r/2}\nabla^2A\|_2^{1/3}+\|A\|_{2,[\gamma\theta>0]}^{1/3}\big),
\end{align*}
and hence
\begin{align*}
    \|\gamma^2\theta^{r/4}A\|_\infty^4\leq c\,\|A\|_{2,[\gamma\theta>0]}^2\big(\|\gamma^4\theta^{r/2}\nabla^2A\|_2^2+\|A\|_{2,[\gamma\theta>0]}^2\big),
\end{align*}
which leads to the result we need to prove.
\end{proof}
\section{Evolution equations}
In this section, we derive the evolution of tensors along Willmore flows.
In particular, those of $\nabla^mA$.
First, as stated in section 2 of \cite{Kuwert2002GradientFF}, we have the following lemmas.
\begin{LEM}\label{eqn2.9}
Let $\phi\in\Gamma((T^\ast\Sigma)^{\otimes(\ell-1)}\otimes N_\Sigma)$, then
\begin{align*}
(\nabla\nabla^\ast-\nabla^\ast\nabla)\phi=A\ast A\ast\phi-(\nabla^\ast T),
\end{align*}
where
\begin{align*}
T(X_0,\ldots,X_\ell)
&=(\nabla_{X_0}\phi)(X_1,X_2,\ldots,X_\ell)-(\nabla_{X_1}\phi)(X_0,X_2,\ldots,X_\ell)\\
&=(\r{R}^{\ell-1}(X_0,X_1)\phi)(X_2,\ldots,X_\ell)\\
&=A\ast A\ast\phi.\tag{Gauss--Codazzi equation}
\end{align*}
\end{LEM}
\begin{CORO}\label{eqn2.11}
\begin{align*}
&(\Laplace\nabla-\nabla\Laplace)\phi=(\nabla\nabla^\ast-\nabla^\ast\nabla)(\nabla\phi)=A\ast A\ast\nabla\phi+A\ast\nabla A\ast\phi,
\end{align*}
and hence
\begin{align*}
(\Laplace\nabla^m-\nabla^m\Laplace)\phi=P_2^m(A)\ast\phi+P_2^{m-1}(A)\ast\nabla\phi+\cdots+P_2^0(A)\ast\nabla^m\phi=\nabla^m(\phi\ast P_2^0).
\end{align*}
\end{CORO}
\begin{LEM}[Simons' identity]\label{eqn2.10}
\begin{align*}
    \Laplace A_{ij}=\nabla^2_{ij}H+g^{k\ell}g^{pq}(\inner{A_{ik}}{A_{jp}}A_{q\ell}-\inner{A_{qk}}{A_{jp}}A_{i\ell}).
\end{align*}
In particular,
\begin{enumerate}[label=(\alph*)]
    \item\label{SimonsA} $\Laplace A=\nabla^2H+A\ast A\ast A$, and
    \item\label{SimonsA0} $\Laplace A^0=S^0(\nabla^2H)+\frac12|H|^2A^0+A^0\ast A^0\ast A^0$, where $S^0(\nabla^2H)_{ij}=\nabla^2_{ij}H-\frac12Hg_{ij}-\frac12(R^\perp H)_{ij}$ denotes the symmetric, trace-free part of $\nabla^2H$.
\end{enumerate}
\end{LEM}
\begin{LEM}\label{eqns1219} 
Letting $V=\partial_tf$ be a normal vector field on $\Sigma$, we have
\begin{enumerate}[label=(\alph*)]
\item\label{eqn2.14} $\partial_t^\perp\nabla_X\phi-\nabla_X\partial_t^\perp\phi=A(X,e_i)\inner{\nabla_{e_i}V}{\phi}+\nabla_{e_i}V\inner{A(X,e_i)}{\phi}=A\ast\nabla V\ast\phi$,
\item\label{eqn2.17} $\partial_t(\nabla_XY)=\big[-\inner{(\nabla_{e_i}A)(X,Y)}{V}+\inner{A(X,Y)}{\nabla_{e_i}V}-\inner{A(X,e_i)}{\nabla_YV}-\inner{A(Y,e_i)}{\nabla_XV}\big]e_i$, and
\item\label{eqn2.18}
$\partial_t^\perp A(X,Y)=\nabla^2_{X,Y}V-A(e_i,X)\inner{A(e_i,Y)}{V}$,
i.e.,
$\partial_t^\perp A=\nabla^2V-(A\swcorner e_i)\otimes\inner{(A\swcorner e_i)}{V}$.
\end{enumerate}
\end{LEM}
In this article, we will let $V=-\theta^r\mathbf{W}(f)$, where $\theta$ is the cutoff function described in section \ref{conventions}.
The following statements are some consequences:
\begin{LEM}\label{evEqn1}
\begin{align*}
(\partial_t^\perp+\theta^r\Laplace^2)A=\theta^r(P_3^2+P_5^0)+\nabla\big(\nabla(\theta^r)\ast(P_1^2+P_3^0)\big).
\end{align*}
\end{LEM}
\begin{proof}
By \ref{eqn2.18} of Lemma \ref{eqns1219},
\begin{align*}
\partial_t^\perp A
&=\nabla^2V+A\ast A\ast V\\
&=-\nabla^2(\theta^r\Laplace H+\theta^r P_3^0)+\theta^r P_2^0\ast(P_1^2+P_3^0)\\
&=\theta^r(-\nabla^2\Laplace H+P_3^2+P_5^0)+\nabla(\theta^r)\ast(P_1^3+P_3^1)+\nabla^2(\theta^r)\ast(P_1^2+P_3^0).
\end{align*}
Also, by Lemma \ref{eqn2.11} and \ref{SimonsA} of Lemma \ref{eqn2.10},
\begin{align*}
\Laplace^2A=\Laplace(\nabla^2H+P_3^0)=\Laplace\nabla^2H+P_3^2=\nabla^2\Laplace H+P_3^2.
\end{align*}
Therefore,
\begin{align*}
(\partial_t^\perp+\theta^r\Laplace^2)A
&=\theta^r(P_3^2+P_5^0)+\nabla(\theta^r)\ast(P_1^3+P_3^1)+\nabla^2(\theta^r)\ast(P_1^2+P_3^0)\\
&=\theta^r(P_3^2+P_5^0)+\nabla\big(\nabla(\theta^r)\ast(P_1^2+P_3^0)\big).
\end{align*}
\end{proof}
\begin{LEM}\label{evEqn2}
If $(\partial_t^\perp+\theta^r\Laplace^2)\phi=Y$ and $\psi=\nabla\phi$, then
\begin{align*}
(\partial_t^\perp+\theta^r\Laplace^2)\psi
&=\nabla Y+\nabla(\theta^r)\ast((P_2^2+P_4^0)\ast\phi+\Laplace^2\phi)+\theta^r\nabla^3(P_2^0\ast\phi).
\end{align*}
\end{LEM}
\begin{proof}
Let $X_1,\ldots,X_\ell$ be time-independent. WLOG, assume $\nabla_{X_k}X_h$ vanishes at a given position and time so that we have
\begin{align*}
(\partial_t^\perp\psi)(X_1,\ldots,X_\ell)
&=\partial_t^\perp\left[(\nabla_{X_1}\phi)(X_2,\ldots,X_\ell)-\sum_{k=2}^\ell\phi(X_2,\ldots,\nabla_{X_1}X_k,\ldots,X_\ell)\right]\\
&=(\partial_t^\perp\nabla_{X_1}\phi)(X_2,\ldots,X_\ell)-\sum_{k=2}^\ell\phi(X_2,\ldots,\partial_t(\nabla_{X_1}X_k),\ldots,X_\ell).
\end{align*}
By \ref{eqn2.14} of Lemma \ref{eqns1219},
\begin{align*}
(\partial_t^\perp\nabla_{X_1}\phi)(X_2,\ldots,X_\ell)=(\nabla_{X_1}\partial_t^\perp\phi)(X_2,\ldots,X_\ell)+A\ast\nabla V\ast\phi.
\end{align*}
Also, by \ref{eqn2.17} of Lemma \ref{eqns1219},
\begin{align*}
\partial_t\nabla_{X_1}X_k=\nabla(A\ast V)\ast(X_1\otimes X_k).
\end{align*}
As a result,
\begin{align*}
\partial_t^\perp\psi=\nabla\partial_t^\perp\phi+\nabla(A\ast V)\ast\phi.
\end{align*}
Next, by Lemma \ref{eqn2.11},
\begin{align*}
\Laplace^2\psi
&=\Laplace^2\nabla\phi\\
&=\Laplace\nabla\Laplace\phi+\Laplace\nabla(P_2^0\ast\phi)\\
&=\nabla\Laplace^2\phi+\nabla(P_2^0\ast\Laplace\phi)+\Laplace\nabla(P_2^0\ast\phi)\\
&=\nabla\Laplace^2\phi+\nabla^3(P_2^0\ast\phi).
\end{align*}
Therefore,
\begin{align*}
&(\partial_t^\perp+\theta^r\Laplace^2)\psi\\
&=\nabla\partial_t^\perp\phi+\nabla(A\ast V)\ast\phi+\theta^r\nabla\Laplace^2\phi+\theta^r\nabla^3(P_2^0\ast\phi)\\
&=\nabla\partial_t^\perp\phi+\nabla(\theta^r\Laplace^2\phi)+\nabla(\theta^r)\ast\Laplace^2\phi+\nabla(A\ast V)\ast\phi+\theta^r\nabla^3(P_2^0\ast\phi)\\
&=\nabla Y+\nabla(\theta^r)\ast\Laplace^2\phi+\nabla(A\ast V)\ast\phi+\theta^r\nabla^3(P_2^0\ast\phi).
\end{align*}
Since $V=-\theta^r\mathbf{W}(f)=\theta^r(P_1^2+P_3^0)$,
\begin{align*}
&\nabla(\theta^r)\ast\Laplace^2\phi+\nabla(A\ast V)\ast\phi+\theta^r\nabla^3(P_2^0\ast\phi)\\
&=\nabla(\theta^r)\ast\Laplace^2\phi+\nabla\big(\theta^r(P_2^2+P_4^0)\big)\ast\phi+\theta^r\nabla^3(P_2^0\ast\phi)\\
&=\nabla(\theta^r)\ast\big((P_2^2+P_4^0)\ast\phi+\Laplace^2\phi\big)+\theta^r\nabla^3(P_2^0\ast\phi).
\end{align*}
\end{proof}
\begin{PROP}\label{evEqn3}
\begin{align*}
(\partial_t^\perp+\theta^r\Laplace^2)(\nabla^mA)=\nabla^m\big(\theta^r(P_3^2+P_5^0)\big)+\nabla^{m+1}\big(\nabla(\theta^r)\ast(P_1^2+P_3^0)\big).
\end{align*}
\end{PROP}
\begin{proof}
We have proved the case for $m=0$ in Lemma \ref{evEqn1}.
Inductively, assume that we have the conclusion for $m-1$.
Let $\phi=\nabla^{m-1}A$ in Lemma \ref{evEqn2} so that
\begin{align*}
(\partial_t^\perp+\theta^r\Laplace^2)(\nabla^mA)
&=\nabla\big((\partial_t^\perp+\theta^r\Laplace^2)(\nabla^{m-1}A)\big)+\nabla(\theta^r)\ast\big((P_2^2+P_4^0)\ast P_1^{m-1}+P_1^{m+3}\big)\\
&\;\;\;\;+\theta^r\nabla^3(P_2^0\ast P_1^{m-1}).
\end{align*}
For the first term:
\begin{align*}
\nabla\big((\partial_t^\perp+\theta^r\Laplace^2)(\nabla^{m-1}A)\big)
&=\nabla\big[\nabla^{m-1}\big(\theta^r(P_3^2+P_5^0)\big)+\nabla^m\big(\nabla(\theta^r)\ast(P_1^2+P_3^0)\big)\big]\\
&=\nabla^m\big(\theta^r(P_3^2+P_5^0)\big)+\nabla^{m+1}\big(\nabla(\theta^r)\ast(P_1^2+P_3^0)\big);
\end{align*}
for the second term:
\begin{align*}
\nabla(\theta^r)\ast\big((P_2^2+P_4^0)\ast P_1^{m-1}+P_1^{m+3}\big)=\nabla^{m+1}\big(\nabla(\theta^r)\ast(P_1^2+P_3^0)\big);
\end{align*}
and for the third term:
\begin{align*}
\theta^r\nabla^3(P_2^0\ast P_1^{m-1})=\nabla^m\big(\theta^r(P_3^2+P_5^0)\big).
\end{align*}
We can derive
\begin{align*}
(\partial_t^\perp+\theta^r\Laplace^2)(\nabla^mA)=\nabla^m\big(\theta^r(P_3^2+P_5^0)\big)+\nabla^{m+1}\big(\nabla(\theta^r)\ast(P_1^2+P_3^0)\big),
\end{align*}
so the claim holds by mathematical induction.
\end{proof}
\section{Energy estimates}
In this section, we estimate the evolution of $L^2$ norms of tensors.
\begin{LEM}\label{evIneq1}
Let $Y=(\partial_t^\perp+\theta^r\Laplace^2)\phi$.
We have
\begin{align*}
&\dv{t}\int_\Sigma\frac12\gamma^s|\phi|^2\dd{\mu}+\int_\Sigma\inner{\Laplace\phi}{\Laplace(\gamma^s\theta^r\phi)}-\inner{Y}{\gamma^s\phi}\dd{\mu}\\
&=\frac12\int_\Sigma\big(\partial_t(\gamma^s)-\gamma^s\inner{H}{V}\big)|\phi|^2\dd{\mu}\\
&\;\;-\int_\Sigma\gamma^s\inner{V}{A(e_{i_k},e_j)}\inner{\phi(e_{i_1},\ldots,e_{i_{r_\phi}})}{\phi(e_{i_1},\ldots,e_{i_{k-1}},e_j,e_{i_{k+1}},\ldots,e_{i_{r_\phi}})}\dd{\mu},
\end{align*}
where $\{e_i\}$ is denotes an orthonormal frame of $\Sigma$, and we use Einstein's conventions on the indices $i_1,\ldots,i_{r_\phi},j\in\{1,2\}$ and $k\in\{1,\ldots,r_\phi\}$.
\end{LEM}
\begin{proof}
\begin{enumerate}[label=(\roman*)]
\item Abusing the notations, let $\{e_i\}$ denote a coordinate system of $\Sigma$ that is orthonormal at the given point and time in $\Sigma\times I$.
Then we have
\begin{align*}
&\partial_t^\perp\big[\phi(e_{i_1},\ldots,e_{i_\ell})\big]\\
&=(Y-\theta^r\Laplace^2\phi)(e_{i_1},\ldots,e_{i_{r_\phi}})+\sum_{k=1}^\ell\phi(e_{i_1},\ldots,e_{i_{k-1}},\partial_t^\top e_{i_k},e_{i_{k+1}},\ldots,e_{i_{r_\phi}}),
\end{align*}
where
\begin{align*}
g(\partial_t^\top e_i,e_j)
=\inner{\partial_te_i}{e_j}
=\inner{D_Ve_j}{e_j}
=\inner{D_{e_i}V}{e_j}
=-\inner{V}{D_{e_i}e_j}
=-\inner{V}{A(e_i,e_j)}.
\end{align*}
Therefore,
\begin{align*}
&\partial_t\left(\frac12\gamma^s|\phi(e_{i_1},\ldots,e_{i_{r_\phi}})|^2\right)\\
&=\gamma^s\inner{\phi(e_{i_1},\ldots,e_{i_{r_\phi}})}{(Y-\theta^r\Laplace^2\phi)(e_{i_1},\ldots,e_{i_{r_\phi}})}\\
&\;\;-\gamma^s\inner{V}{A(e_{i_k},e_j)}\inner{\phi(e_{i_1},\ldots,e_{i_{r_\phi}})}{\phi(e_{i_1},\ldots,e_{i_{k-1}},e_j,e_{i_{k+1}},\ldots,e_{i_{r_\phi}})}\\
&\;\;+\frac12\partial_t(\gamma^s)|\phi(e_{i_1},\ldots,e_{i_{r_\phi}})|^2.
\end{align*}
That is,
\begin{align*}
&\partial_t\left(\frac12\gamma^s|\phi|^2\right)\\
&=\inner{Y}{\gamma^s\phi}-\theta^r\inner{\Laplace^2\phi}{\gamma^s\phi}+\frac12\partial_t(\gamma^s)|\phi|^2\\
&\phantom{==}-\gamma^s\inner{V}{A(e_{i_k},e_j)}\inner{\phi(e_{i_1},\ldots,e_{i_{r_\phi}})}{\phi(e_{i_1},\ldots,e_{i_{k-1}},e_j,e_{i_{k+1}},\ldots,e_{i_{r_\phi}})}.
\end{align*}
In summary,
\begin{align*}
&\dv{t}\int_\Sigma\frac12\gamma^s|\phi|^2\dd{\mu}\\
&=\int_\Sigma\bigg(\inner{Y}{\gamma^s\phi}-\theta^r\inner{\Laplace^2\phi}{\gamma^s\phi}+\frac12\partial_t(\gamma^s)|\phi|^2-\frac12\gamma^s\inner{H}{V}|\phi|^2\\
&\phantom{==}-\gamma^s\inner{V}{A(e_{i_k},e_j)}\inner{\phi(e_{i_1},\ldots,e_{i_{r_\phi}})}{\phi(e_{i_1},\ldots,e_{i_{k-1}},e_j,e_{i_{k+1}},\ldots,e_{i_{r_\phi}})}\bigg)\dd{\mu}\\
&=\int_\Sigma\bigg(\inner{Y}{\gamma^s\phi}-\inner{\Laplace\phi}{\Laplace(\gamma^s\theta^r\phi)}+\frac12\partial_t(\gamma^s)|\phi|^2-\frac12\gamma^s\inner{H}{V}|\phi|^2\\
&\phantom{==}-\gamma^s\inner{V}{A(e_{i_k},e_j)}\inner{\phi(e_{i_1},\ldots,e_{i_{r_\phi}})}{\phi(e_{i_1},\ldots,e_{i_{k-1}},e_j,e_{i_{k+1}},\ldots,e_{i_{r_\phi}})}\bigg)\dd{\mu}.
\end{align*}
\end{enumerate}
\end{proof}
\begin{LEM}\label{evIneq2}
Again let $Y=(\partial_t^\perp+\theta^r\Laplace^2)\phi$.
If $s\geq4$, $r\geq4$, we have
\begin{align*}
    &\dv{t}\int_\Sigma\gamma^s|\phi|^2\dd{\mu}+\frac78\int_\Sigma\gamma^s\theta^r|\nabla^2\phi|^2\dd{\mu}\\
    &\leq c\int_\Sigma\gamma^s\phi\ast\big(Y+A\ast\phi\ast V+\theta^r|\nabla A|^2\phi+\theta^r|A|^4\phi\big)\dd{\mu}+c\,K^4\int_\Sigma\gamma^{s-4}\theta^{r-4}|\phi|^2\dd{\mu},
\end{align*}
where $c=c(s,r,r_\phi)$.
\end{LEM}
\begin{proof}
\begin{enumerate}[label=(\roman*)]
\item
By Lemma \ref{evIneq1}, we have
\begin{align*}
&\dv{t}\int_\Sigma\gamma^s|\phi|^2\dd{\mu}+2\int_\Sigma\inner{\Laplace\phi}{\Laplace(\gamma^s\theta^r\phi)}-\gamma^s\inner{Y}{\phi}\dd{\mu}\\
&=\int_\Sigma\bigg(s\gamma^{s-1}\big(\partial_t\gamma-\gamma\inner{H}{V}\big)|\phi|^2\\
&\phantom{==}-2\gamma^s\sum_{k=1}^\ell\inner{V}{A(e_{i_k},e_j)}\inner{\phi(e_{i_1},\ldots,e_{i_\ell})}{\phi(e_{i_1},\ldots,e_{i_{k-1}},e_j,e_{i_{k+1}},\ldots,e_{i_\ell})}\bigg)\dd{\mu}\\
&\leq c\int_\Sigma\gamma^{s-1}\partial_t\gamma|\phi|^2\dd{\mu}+c\int_\Sigma\gamma^sA\ast\phi\ast\phi\ast V\dd{\mu},
\end{align*}
so that
\begin{align*}
    &\dv{t}\int_\Sigma\gamma^s|\phi|^2\dd{\mu}\\
    &=-2\int_\Sigma\inner{\Laplace\phi}{\Laplace(\gamma^s\theta^r\phi)}\dd{\mu}+c\int_\Sigma\gamma^{s-1}\partial_t\gamma|\phi|^2\dd{\mu}+c\int_\Sigma\gamma^s\phi\ast(Y+A\ast\phi\ast V)\dd{\mu}.
\end{align*}
\item
Using integration by parts and Cauchy-Schwarz inequality,
\begin{align*}
    &-\int_\Sigma\inner{\Laplace\phi}{\Laplace(\gamma^s\theta^r\phi)}\dd{\mu}\\
    &=\int_\Sigma\inner{\nabla\Laplace\phi}{\nabla(\gamma^s\theta^r\phi)}\dd{\mu}\\
    &=\int_\Sigma\inner{(\Laplace\nabla\phi+A\ast A\ast\nabla\phi+A\ast\nabla A\ast\phi)}{\nabla(\gamma^s\theta^r\phi)}\dd{\mu}\\
    &\leq-\int_\Sigma\inner{\nabla^2\phi}{\nabla^2(\gamma^s\theta^r\phi)}\dd{\mu}\\
    &\phantom{==}+c\int_\Sigma\big(|A|^2|\nabla\phi|+|A|\,|\nabla A|\,|\phi|\big)\big(\gamma^{s-1}\theta^{r-1}K|\phi|+\gamma^s\theta^r|\nabla\phi|\big)\dd{\mu}\\
    &\leq-\int_\Sigma\inner{\nabla^2\phi}{\gamma^s\theta^r\nabla^2\phi}\dd{\mu}\\
    &\phantom{==}+c\int_\Sigma\big(\gamma^{s-1}\theta^{r-1}K|\nabla\phi|+\gamma^{s-2}\theta^{r-2}K^2|\phi|+\gamma^{s-1}\theta^{r-1}K|A|\,|\phi|\big)|\nabla^2\phi|\dd{\mu}\\
    &\phantom{==}+c\int_\Sigma\big(|A|^2|\nabla\phi|+|A|\,|\nabla A|\,|\phi|\big)\big(\gamma^{s-1}\theta^{r-1}K|\phi|+\gamma^s\theta^r|\nabla\phi|\big)\dd{\mu}\\
    &\leq-\int_\Sigma\gamma^s\theta^r|\nabla^2\phi|^2\dd{\mu}+\frac14\int_\Sigma\gamma^s\theta^r|\nabla^2\phi|^2\dd{\mu}
    +c\int_\Sigma\big(\gamma^s\theta^r|A|^2|\nabla\phi|^2+\gamma^s\theta^r|\nabla A|^2|\phi|^2\big)\dd{\mu}\\
    &\phantom{==}+c\int_\Sigma\big(\gamma^{s-2}\theta^{r-2}K^2|\nabla\phi|^2+\gamma^{s-4}\theta^{r-4}K^4|\phi|^2+\gamma^{s-2}\theta^{r-2}K^2|A|^2|\phi|^2\big)\dd{\mu}\\
    &\leq-\frac34\int_\Sigma\gamma^s\theta^r|\nabla^2\phi|^2\dd{\mu}
    +c\int_\Sigma\big(\gamma^s\theta^r|A|^2|\nabla\phi|^2+\gamma^s\theta^r|\nabla A|^2|\phi|^2\big)\dd{\mu}\\
    &\phantom{==}+c\int_\Sigma\big(\gamma^{s-2}\theta^{r-2}K^2|\nabla\phi|^2+\gamma^{s-4}\theta^{r-4}K^4|\phi|^2+\gamma^s\theta^r|A|^4|\phi|^2\big)\dd{\mu},
\end{align*}
so that the conclusion in (i) becomes
\begin{align*}
    &\dv{t}\int_\Sigma\gamma^s|\phi|^2\dd{\mu}+\frac32\int_\Sigma\gamma^s\theta^r|\nabla^2\phi|^2\dd{\mu}\\
    &\leq c\int_\Sigma\gamma^{s-1}\partial_t\gamma|\phi|^2\dd{\mu}
    +c\int_\Sigma\gamma^s\phi\ast(Y+A\ast\phi\ast V+\theta^r|\nabla A|^2\phi+\theta^r|A|^4\phi)\dd{\mu}\\
    &\phantom{==}+c\int_\Sigma\big(\gamma^s\theta^r|A|^2|\nabla\phi|^2+\gamma^{s-2}\theta^{r-2}K^2|\nabla\phi|^2+\gamma^{s-4}\theta^{r-4}K^4|\phi|^2\big)\dd{\mu}.
\end{align*}
\item Using integration by parts,
\begin{align*}
&\int_\Sigma\gamma^s\theta^r|A|^2|\nabla\phi|^2\dd{\mu}\\
&=\int_\Sigma\inner{\phi}{\nabla^\ast\left(\gamma^s\theta^r|A|^2\nabla\phi\right)}\dd{\mu}\\
&\leq c\int_\Sigma\gamma^s\theta^r|A|^2|\phi|\,|\nabla^2\phi|+\gamma^{s-1}\theta^{r-1}K|A|^2|\phi|\,|\nabla\phi|+\gamma^s\theta^r|A|\,|\nabla A|\,|\phi|\,|\nabla\phi|\dd{\mu}
\end{align*}\begin{align*}
&\leq\varepsilon\int_\Sigma\gamma^s\theta^r|\nabla^2\phi|^2\dd{\mu}+\varepsilon\int_\Sigma\gamma^{s-2}\theta^{r-2}K^2|\nabla\phi|^2\dd{\mu}+c\,\varepsilon^{-1}\int_\Sigma\gamma^s\theta^r|A|^4|\phi|^2\dd{\mu}\\
&\phantom{==}+\varepsilon\int_\Sigma\gamma^s\theta^r|A|^2|\nabla\phi|^2\dd{\mu}+c\,\varepsilon^{-1}\int_\Sigma\gamma^s\theta^r|\nabla A|^2|\phi|^2\dd{\mu}
\end{align*}
for any $\varepsilon>0$, and hence by taking $\varepsilon$ to be sufficiently small,
\begin{align*}
    &\dv{t}\int_\Sigma\gamma^s|\phi|^2\dd{\mu}+\int_\Sigma\gamma^s\theta^r|\nabla^2\phi|^2\dd{\mu}\\
    &\leq c\int_\Sigma\gamma^{s-1}\partial_t\gamma|\phi|^2\dd{\mu}+\int_\Sigma\gamma^s\phi\ast(Y+A\ast\phi\ast V+\theta^r|\nabla A|^2\phi+\theta^r|A|^4\phi)\dd{\mu}\\
    &\phantom{==}+c\int_\Sigma\big(\gamma^{s-2}\theta^{r-2}K^2|\nabla\phi|^2+\gamma^{s-4}\theta^{r-4}K^4|\phi|^2\big)\dd{\mu}.
\end{align*}
\item Next, since $V=-\theta^r(\Laplace H+A\ast A\ast A)$,
\begin{align*}
    &\int_\Sigma\gamma^{s-1}\partial_t\gamma|\phi|^2\dd{\mu}
    =\int_\Sigma\gamma^{s-1}D\widehat\gamma(-\theta^r\Laplace H+\theta^rP_3^0)|\phi|^2\dd{\mu}.
\end{align*}
On one hand, by Young's inequality,
\begin{align*}
&\int_\Sigma\gamma^{s-1}\theta^rD\widehat\gamma P_3^0|\phi|^2\dd{\mu}\\
&\leq c\int_\Sigma\gamma^{s-1}\theta^rK|A|^3|\phi|^2\dd{\mu}\\
&=c\int_\Sigma\big(\gamma^{3s/4}\theta^{3r/4}|A|^3|\phi|^{3/2}\big)\big(\gamma^{s/4-1}\theta^{r/4}K|\phi|^{1/2}\big)\dd{\mu}\\
&\leq c\int_\Sigma\big(\gamma^s\theta^r|A|^4|\phi|^2+\gamma^{s-4}\theta^rK^4|\phi|^2\big)\dd{\mu}.
\end{align*}
On the other hand, with integration by parts,
\begin{align*}
&-\int_\Sigma\gamma^{s-1}\theta^r|\phi|^2D\widehat\gamma(\Laplace H)\dd{\mu}\\
&=-\int_\Sigma\gamma^{s-1}\theta^r|\phi|^2D\widehat\gamma\big(D_{e_i}\nabla_{e_i}H-\inner{A(e_i,e_j)}{\nabla_{e_j}H}D_{e_i}f\big)\dd{\mu}\\
&=\int_\Sigma\big[\gamma^{s-1}\theta^r|\phi|^2D^2\widehat\gamma(D_{e_i}f,\nabla_{e_i}H)+\inner{D_{e_i}\big(\gamma^{s-1}\theta^r|\phi|^2\big)}{(D\widehat\gamma(\nabla_{e_i}H)}\\
&\phantom{===}+\gamma^{s-1}\theta^r|\phi|^2\inner{A(e_i,e_j)}{\nabla_{e_j}H}D\widehat\gamma(D_{e_i}f)\big]\dd{\mu}\\
&\leq c\int_\Sigma\gamma^{s-2}\theta^{r-1}\big[\gamma\theta K^2|\nabla A|\,|\phi|^2+(\gamma\theta|\nabla\phi|+K|\phi|)K|\nabla A|\,|\phi|+\gamma\theta K|A|\,|\nabla A|\,|\phi|^2\big]\dd{\mu}\\
&\leq c\int\gamma^{s-2}\theta^{r-1}\big[K^2|\nabla A|\,|\phi|^2+\gamma\theta K|\nabla A|\,|\phi|\,|\nabla\phi|+\gamma\theta K|A|\,|\nabla A|\,|\phi|^2\big]\dd{\mu}\\
&\leq c\int_\Sigma\big[\gamma^s\theta^r|\nabla A|^2|\phi|^2+\gamma^{s-4}\theta^{r-4}K^4|\phi|^2+\gamma^{s-2}\theta^rK^2|\nabla\phi|^2+\gamma^{s-2}\theta^rK^2|A|^2|\phi|^2\big]\dd{\mu}\\
&\leq c\int_\Sigma\big[\gamma^s\theta^r|\nabla A|^2|\phi|^2+\gamma^s\theta^r|A|^4|\phi|^2+\gamma^{s-4}\theta^{r-4}K^4|\phi|^2+\gamma^{s-2}\theta^rK^2|\nabla\phi|^2\big]\dd{\mu}.
\end{align*}
As a result,
\begin{align*}
    &\dv{t}\int_\Sigma\gamma^s|\phi|^2\dd{\mu}+\int_\Sigma\gamma^s\theta^r|\nabla^2\phi|^2\dd{\mu}\\
    &\leq c\int_\Sigma\gamma^s\phi\ast(Y+A\ast\phi\ast V+\theta^r|\nabla A|^2\phi+\theta^r|A|^4\phi)\dd{\mu}\\
    &\phantom{==}+c\int_\Sigma\big(\gamma^{s-2}\theta^{r-2}K^2|\nabla\phi|^2+\gamma^{s-4}\theta^{r-4}K^4|\phi|^2\big)\dd{\mu}.
\end{align*}
\item
Finally, for arbitrary $\varepsilon>0$,
\begin{align*}
    &\int_\Sigma\gamma^{s-2}\theta^{r-2}K^2|\nabla\phi|^2\dd{\mu}\\
    &\leq c\int_\Sigma\big(\gamma^{s-2}\theta^{r-2}K^2|\nabla^2\phi|+\gamma^{s-3}\theta^{r-3}K^3|\nabla\phi|\big)|\phi|\dd{\mu}\\
    &\leq\varepsilon\int_\Sigma\big(\gamma^s\theta^r|\nabla^2\phi|^2+\gamma^{s-2}\theta^{r-2}K^2|\nabla\phi|^2\big)\dd{\mu}+c\,\varepsilon^{-1}\int_\Sigma\gamma^{s-4}\theta^{r-4}K^4|\phi|^2\dd{\mu},
\end{align*}
so that by taking $\varepsilon$ to be sufficiently small,
\begin{align*}
    &\dv{t}\int_\Sigma\gamma^s|\phi|^2\dd{\mu}+\frac78\int_\Sigma\gamma^s\theta^r|\nabla^2\phi|^2\dd{\mu}\\
    &=c\int_\Sigma\gamma^s\phi\ast(Y+A\ast\phi\ast V+\theta^r|\nabla A|^2\phi+\theta^r|A|^4\phi)\dd{\mu}+c\int_\Sigma\gamma^{s-4}\theta^{r-4}K^4|\phi|^2\dd{\mu},
\end{align*}
which is what we need to prove.
\end{enumerate}
\end{proof}
\begin{PROP}\label{evIneq3} 
Let $0\leq k\leq m$.
If $s,r\geq2k+4$, we have
\begin{align*}
&\dv{t}\int_\Sigma\gamma^s|\nabla^mA|^2\dd{\mu}+\frac34\int_\Sigma\gamma^s\theta^r|\nabla^{m+2}A|^2\dd{\mu}\\
&\leq c\int_\Sigma\gamma^s\nabla^mA\ast\big[\nabla^m\big(\theta^r(P_3^2+P_5^0)\big)+\nabla^{m+1}\big(\theta^{r-1}\nabla\theta\ast(P_1^2+P_3^0)\big)\big]\dd{\mu}\\
&\phantom{==}+c\,K^{4+2k}\int_\Sigma\gamma^{s-4-2k}\theta^{r-4-2k}|\nabla^{m-k}A|^2\dd{\mu},
\end{align*}
where $c=c(s,r,m)$.
\end{PROP}
\begin{proof}
By Proposition \ref{evEqn3}, we have
\begin{align*}
Y=(\partial_t^\perp+\theta^r\Laplace^2)\phi
=\nabla^m\big(\theta^r(P_3^2+P_5^0)\big)+\nabla^{m+1}\big(\theta^{r-1}\nabla\theta\ast(P_1^2+P_3^0)\big).
\end{align*}
In addition, $V=\theta^r(P_1^2+P_3^0)$ implies
\begin{align*}
&A\ast\nabla^mA\ast V+\theta^r|\nabla A|^2\nabla^mA+\theta^r|A|^4\nabla^mA
=\theta^r(P_3^{m+2}+P_5^m).
\end{align*}
Therefore, by taking $\phi=\nabla^mA$ in Lemma \ref{evIneq2},
\begin{align*}
    &\dv{t}\int_\Sigma\gamma^s|\nabla^mA|^2+\frac78\int_\Sigma\gamma^s\theta^r|\nabla^{m+2}A|^2\dd{\mu}\\
    &\leq c\int_\Sigma\gamma^s\nabla^mA\ast\big(Y+A\ast\nabla^mA\ast V+\theta^r|\nabla A|^2\nabla^mA+\theta^r|A|^4\nabla^mA\big)\dd{\mu}\\
    &\phantom{==}+c\,K^4\int_\Sigma\gamma^{s-4}\theta^{r-4}|\nabla^mA|^2\dd{\mu}\\
    &\leq c\int_\Sigma\gamma^s\nabla^mA\ast\big[\nabla^m\big(\theta^r(P_3^2+P_5^0)\big)+\nabla^{m+1}\big(\theta^{r-1}\nabla\theta\ast(P_1^2+P_3^0)\big)\big]\dd{\mu}\\
    &\phantom{==}+c\,K^4\int_\Sigma\gamma^{s-4}\theta^{r-4}|\nabla^mA|^2\dd{\mu}.
\end{align*}
If $k>0$, by Proposition \ref{interpolation1},
\begin{align*}
&K^4\int_\Sigma\gamma^{s-4}\theta^{r-4}|\nabla^mA|^2\dd{\mu}\\
&\leq\varepsilon\int_\Sigma\gamma^s\theta^r|\nabla^{m+2}A|^2\dd{\mu}+c(s,r,m,\varepsilon)\,K^{4+2k}\int_\Sigma\gamma^{s-4-2k}\theta^{r-4-2k}|\nabla^{m-k}A|^2\dd{\mu},
\end{align*}
and hence by taking $\varepsilon$ to be sufficiently small,
\begin{align*}
&\dv{t}\int_\Sigma|\nabla^mA|^2\gamma^s\dd{\mu}+\frac34\int_\Sigma\gamma^s\theta^r|\nabla^{m+2}A|^2\dd{\mu}\\
&\leq c\int_\Sigma\gamma^s\nabla^mA\ast\big[\nabla^m\big(\theta^r(P_3^2+P_5^0)\big)+\nabla^{m+1}\big(\theta^{r-1}\nabla\theta\ast(P_1^2+P_3^0)\big)\big]\dd{\mu}\\
&\phantom{==}+c\,K^{4+2k}\int_{[\gamma\theta>0]}\gamma^{s-4-2k}\theta^{r-4-2k}|\nabla^{m-k}A|^2\dd{\mu}.
\end{align*}
\end{proof}
\begin{LEM}\label{bootstrapLower}
Let $f:\Sigma\times[0,T)\to\bR^n$ be a solution to the modified equation (\ref{pdeCuttoff}).
If $s,r\geq4$, then we can choose $\varepsilon_0$ so that assuming (\ref{energyNonconcentrating}) for all $0\leq t\leq t_0$ for some $0<t_0<T$, namely
\begin{align}
    \sup_{0\leq t\leq t_0} \int_{[\gamma\theta>0]}|A|^2\dd{\mu}\leq\varepsilon_0,\label{energyNonconcentratingST}
\end{align}
we have
\begin{align*}
\int_{[\gamma=1]}|A|^2\dd{\mu}+\frac12\int_0^t\int_{[\gamma=1]}\theta^r\big(|\nabla^2A|^2+|A|^6\big)\dd{\mu}\dd{t'}\leq\int_{[\gamma>0]}|A|^2\dd{\mu}\bigg|_{t=0}+c\,K^4\varepsilon_0t
\end{align*}
for all $t\in[0,t_0)$, where $c=c(n,s,r)$.
\end{LEM}
\begin{proof}
By Proposition \ref{evIneq3} (with $m=0$ and $k=0$),
\begin{align*}
&\dv{t}\int_\Sigma\gamma^s|A|^2\dd{\mu}+\frac34\int_\Sigma\gamma^s\theta^r|\nabla^2A|^2\dd{\mu}+\frac34\int_\Sigma\gamma^s\theta^r|A|^6\dd{\mu}\\
&\leq c\int_\Sigma\gamma^sA\ast\big[\theta^r(P_3^2+P_5^0)+\nabla\big(\theta^{r-1}\nabla\theta\ast(P_1^2+P_3^0)\big)\big]\dd{\mu}\\
&\phantom{\leq c\,} +c\,K^4\int_\Sigma\gamma^{s-4}\theta^{r-4}|A|^2\dd{\mu}+\frac34\int_\Sigma\gamma^s\theta^r|A|^6\dd{\mu}
\end{align*}\begin{align*}
&\leq c\int_\Sigma\gamma^s\theta^rA\ast(P_3^2+P_5^0)+\nabla(\gamma^sA)\ast(\theta^{r-1}\nabla\theta)\ast(P_1^2+P_3^0)\dd{\mu}\\
&\phantom{\leq c\,} +c\,K^4\int_\Sigma\gamma^{s-4}\theta^{r-4}|A|^2\dd{\mu}+\frac34\int_\Sigma\gamma^s\theta^r|A|^6\dd{\mu}\\
&\leq c\int_\Sigma\gamma^s\theta^r\big(|\nabla^2A|\,|A|^3+|\nabla A|^2|A|^2+|A|^6\big)\dd{\mu}\\
&\phantom{\leq c\,} +c\int_\Sigma\big(\gamma^s\theta^{r-1}K|\nabla A|+\gamma^{s-1}\theta^{r-1}K^2|A|\big)\big(|\nabla^2A|+|A|^3\big)\dd{\mu}+c\,K^4\varepsilon_0\\
&\leq\frac1{12}\int_\Sigma\gamma^s\theta^r|\nabla^2A|^2\dd{\mu}+c\int_\Sigma\gamma^s\theta^r\big(|\nabla A|^2|A|^2+|A|^6\big)\dd{\mu}+c\,K^2\int_\Sigma\gamma^s\theta^{r-2}|\nabla A|^2\dd{\mu}+c\,K^4\varepsilon_0.
\end{align*}
Therefore by Proposition \ref{interpolation1} and Lemma \ref{sobolev42coro},
\begin{align*}
&\dv{t}\int_\Sigma\gamma^s|A|^2\dd{\mu}+\frac34\int_\Sigma\gamma^s\theta^r(|\nabla^2A|^2+|A|^6)\dd{\mu}\\
&\leq\frac16\int_\Sigma\gamma^s\theta^r|\nabla^2A|^2\dd{\mu}+c\int_{[\gamma\theta>0]}|A|^2\dd{\mu}\int_\Sigma\gamma^s\theta^r\big(|\nabla^2A|^2+|A|^6\big)\dd{\mu}+c\,K^4\varepsilon_0+c\,K^4\varepsilon_0^2\\
&\leq\frac14\int_\Sigma\gamma^s\theta^r\big(|\nabla^2A|^2+|A|^6\big)\dd{\mu}+c\,K^4\varepsilon_0,
\end{align*}
and hence
\begin{align*}
\int_\Sigma\gamma^s|A|^2\dd{\mu}+\frac12\int_0^t\int_\Sigma\gamma^s\theta^r\big(|\nabla^2A|^2+|A|^6\big)\dd{\mu}\dd{t'}\leq\left(\int_\Sigma\gamma^s\theta^r|A|^2\dd{\mu}\right)\bigg|_{t=0}+c\,K^4\varepsilon_0t.
\end{align*}
\end{proof}
\begin{PROP}\label{evIneq4}
Let $m\geq1$.
If $s\geq6$ and $r\geq20$, then we can choose $\varepsilon_0$ so that assuming (\ref{energyNonconcentrating}), we have that for some $c=c(n,s,r,m,K)$,
\begin{align*}
    &\dv{t}\int_\Sigma\gamma^s|\nabla^mA|^2\dd{\mu}+\frac12\int_\Sigma\gamma^s\theta^r|\nabla^{m+2}A|^2\dd{\mu}
    \leq c\,\|\theta^{r/4}A\|_{\infty,[\gamma>0]}^4\int_{[\theta>0]}\gamma^s|\nabla^mA|^2\dd{\mu}+\beta_{m,\gamma},
\end{align*}
where
\begin{align*}
    &\beta_{1,\gamma}=c\,\|A\|_{2,[\gamma\theta>0]}^2,\\
    &\beta_{2,\gamma}=c\,\big(1+\|\theta^{r/4}A\|_{\infty,[\gamma>0]}^4\big)\|A\|_{2,[\gamma\theta>0]}^2,
\end{align*}
and when $m\geq3$, $\beta_{m,\gamma}$ only depends on $n$, $s$, $r$, $m$, $K$, and
\begin{align*}
    \|\nabla^jA\|_{p,[\gamma\theta>0]}\text{, where either }
    \begin{cases}
    j=0,\ldots,m-2,\\
    p=2,\ldots,2m+4,
    \end{cases}
    \text{ or }
    \begin{cases}
    j=0,1,\\
    p=\infty.
    \end{cases}
\end{align*}
\end{PROP}
\begin{proof}
With $c=c(n,s,r,m,\varepsilon)$, we have
\begin{enumerate}[label=(\roman*)]
\item 
For $m=1$,
\begin{align*}
    &\int_\Sigma(\gamma^s\nabla A)\ast\big[\nabla\big(\theta^r(P_3^2+P_5^0)\big)+\nabla^2\big(\theta^{r-1}\nabla\theta\ast(P_1^2+P_3^0)\big)\big]\dd{\mu}
\\
    &\leq\int_\Sigma(\gamma^s\nabla A)\ast\nabla\big(\theta^r(P_3^2+P_5^0)\big)+\nabla^2(\gamma^s\nabla A)\ast\big(\theta^{r-1}\nabla\theta\ast(P_1^2+P_3^0)\big)\dd{\mu}
\\
    &\leq\int_\Sigma\gamma^s|\nabla A|\cdot\big[\theta^r(|\nabla^3A|\,|A|^2+|\nabla^2A|\,|\nabla A|\,|A|+|\nabla A|^3+|\nabla A|\,|A|^4)\\
    &\phantom{===}+\theta^{r-1}K(|\nabla^2A|\,|A|^2+|\nabla A|^2|A|+|A|^5)\big]\\
    &\phantom{===}+\big[\gamma^s|\nabla^3A|+\gamma^{s-1}K|\nabla^2A|+(\gamma^{s-2}K^2+\gamma^{s-1}K|A|)|\nabla A|\big]\cdot\theta^{r-1}K(|\nabla^2A|+|A|^3)\dd{\mu}
\\
    &\leq\varepsilon\int_\Sigma\gamma^s\theta^r|\nabla^3A|^2\dd{\mu}
    +c\int_\Sigma(\gamma^{s-2}\theta^{r-2}K^2+\gamma^s\theta^r|A|^2)|\nabla^2A|^2\dd{\mu}\\
    &\phantom{==}+c\int_\Sigma(\gamma^{s-4}\theta^{r-4}K^4+\gamma^s\theta^r|A|^4)|\nabla A|^2\dd{\mu}+c\int_\Sigma\gamma^s\theta^{r-2}K^2|A|^6\dd{\mu}+c\int_\Sigma\gamma^s\theta^r|\nabla A|^4\dd{\mu}
\\
    &\leq\varepsilon\int_\Sigma\gamma^s\theta^r|\nabla^3A|^2\dd{\mu}
    +c\int_\Sigma(\gamma^{s-2}\theta^{r-2}K^2+\gamma^s\theta^r|A|^2)|\nabla^2A|^2\dd{\mu}\\
    &\phantom{==}+c\int_\Sigma(\gamma^{s-4}\theta^{r-4}K^4+\gamma^s\theta^r|A|^4)|\nabla A|^2\dd{\mu}+c\int_\Sigma\gamma^s\theta^r|\nabla A|^4\dd{\mu}+c\,K^6\|A\|_{2,[\gamma\theta>0]}^2\tag{Lemma \ref{sobolev42coro}}
\\
    &\leq\varepsilon\int_\Sigma\gamma^s\theta^r|\nabla^3A|^2\dd{\mu}
    +c\int_\Sigma\big(\gamma^{s-2}\theta^{r-2}K^2+\gamma^s\theta^{r/2}\|\theta^{r/4}A\|_{\infty,[\gamma>0]}^2\big)|\nabla^2A|^2\dd{\mu}\\
    &\phantom{==}+c\int_{[\theta>0]}\big(\gamma^{s-4}\theta^{r-4}K^4+\gamma^s\big\|\theta^{r/4}A\big\|_{\infty,[\gamma>0]}^4\big)|\nabla A|^2\dd{\mu}+c\,K^6\|A\|_{2,[\gamma\theta>0]}^2\tag{Proposition \ref{interpolation2}}
\\
    &\leq2\varepsilon\int_\Sigma\gamma^s\theta^r|\nabla^3A|^2\dd{\mu}
    +c\int_\Sigma\gamma^{s-2}\theta^{r-2}K^2|\nabla^2A|^2\dd{\mu}\\
    &\phantom{==}+c\int_{[\theta>0]}\big(\gamma^{s-4}\theta^{r-4}K^4+\gamma^s\|\theta^{r/4}A\|_{\infty,[\gamma>0]}^4\big)|\nabla A|^2\dd{\mu}+c\,K^6\|A\|_{2,[\gamma\theta>0]}^2\tag{Proposition \ref{interpolation3}}
\\
    &\leq3\varepsilon\int_\Sigma\gamma^s\theta^r|\nabla^3A|^2\dd{\mu}+c\,\|\theta^{r/4}A\|_{\infty,[\gamma>0]}^4\int_{[\theta>0]}\gamma^s|\nabla A|^2\dd{\mu}+c\,K^6\|A\|_{2,[\gamma\theta>0]}^2.\tag{Proposition \ref{interpolation1}}
\end{align*}
Finally, apply this estimate with sufficiently small $\varepsilon$ and $k=1$ in Proposition \ref{evIneq3} so that
\begin{align*}
    &\dv{t}\int_\Sigma\gamma^s|\nabla A|^2\dd{\mu}+\frac12\int_\Sigma\gamma^s\theta^r|\nabla^3A|^2\dd{\mu}\\
    &\leq c\,\|\theta^{r/4}A\|_{\infty,[\gamma>0]}^4\int_{[\theta>0]}\gamma^s|\nabla A|^2\dd{\mu}+c\,K^6\|A\|_{2,[\gamma\theta>0]}^2.
\end{align*}
\item
For $m=2$,
\begin{align*}
    &\int_\Sigma\gamma^s\nabla^2A\ast\big[\nabla^2\big(\theta^r(P_3^2+P_5^0)\big)+\nabla^3\big(\theta^{r-1}\nabla\theta\ast(P_1^2+P_3^0)\big)\big]\dd{\mu}
\\
    &\leq\int_\Sigma\Big(\nabla^2(\gamma^s\nabla^2A)\ast\big[\theta^rP_3^2+\nabla\big(\theta^{r-1}\nabla\theta\ast(P_1^2+P_3^0)\big)\big]+\gamma^s\nabla^2A\ast\nabla^2\big(\theta^rP_5^0\big)\Big)\dd{\mu}
\\
    &\leq c\int_\Sigma\Big(\big[\gamma^s|\nabla^4A|+\gamma^{s-1}K(|\nabla^3A|+|\nabla^2A|\,|A|)+\gamma^{s-2}K^2|\nabla^2A|\big]\\
    &\phantom{===}\cdot\big[\theta^r(|\nabla^2A|\,|A|^2+|\nabla A|^2\,|A|)+\theta^{r-1}K(|\nabla^3A|+|\nabla^2A|\,|A|+|\nabla A|\,|A|^2+|A|^4)\\
    &\phantom{===}+\theta^{r-2}K^2(|\nabla^2A|+|A|^3)\big]
    +\gamma^s|\nabla^2A|\,\big[\theta^r(|\nabla^2A|\,|A|^4+|\nabla A|^2\,|A|^3)\\
    &\phantom{===}+\theta^{r-1}K(|\nabla A|\,|A|^4+|A|^6)
    +\theta^{r-2}K^2|A|^5\big]\Big)\dd{\mu}
\\
    &\leq\varepsilon\int_\Sigma\gamma^s\theta^r|\nabla^4A|^2\dd{\mu}
    +c\int_\Sigma(\gamma^{s-2}\theta^{r-2}K^2+\gamma^s\theta^r|A|^2)|\nabla^3A|^2\dd{\mu}\\
    &\phantom{==}+c\int_\Sigma(\gamma^{s-4}\theta^{r-4}K^4+\gamma^{s-2}\theta^{r-2}K^2|A|^2+\gamma^s\theta^r|A|^4)|\nabla^2A|^2\dd{\mu}\\
    &\phantom{==}+c\int_\Sigma\gamma^s\theta^{r-2}K^2|\nabla A|^2|A|^4\dd{\mu}
    +c\int_\Sigma\gamma^s\theta^{r-4}K^4|A|^6\dd{\mu}
    +c\int_\Sigma\gamma^s\theta^{r-2}K^2|A|^8\dd{\mu}\\
    &\phantom{==}+c\int_\Sigma\gamma^s\theta^r|\nabla A|^4|A|^2\dd{\mu}
\\
    &\leq\varepsilon\int_\Sigma\gamma^s\theta^r|\nabla^4A|^2\dd{\mu}
    +c\int_\Sigma(\gamma^{s-2}\theta^{r-2}K^2+\gamma^s\theta^r|A|^2)|\nabla^3A|^2\dd{\mu}\\
    &\phantom{==}+c\int_\Sigma(\gamma^{s-4}\theta^{r-4}K^4+\gamma^s\theta^r|A|^4)|\nabla^2A|^2\dd{\mu}+c\int_\Sigma\gamma^s\theta^r|\nabla A|^4|A|^2\dd{\mu}\\
    &\phantom{==}+c\int_\Sigma\gamma^s\theta^{r-8}K^8|A|^2\dd{\mu}+c\int_\Sigma\gamma^s\theta^{r-2}K^2|A|^8\dd{\mu}
\\
    &\leq\varepsilon\int_\Sigma\gamma^s\theta^r|\nabla^4A|^2\dd{\mu}
    +c\int_\Sigma\big(\gamma^{s-2}\theta^{r-2}K^2+\gamma^s\theta^{r/2}\|\theta^{r/4}A\|_{\infty,[\gamma>0]}^2\big)|\nabla^3A|^2\dd{\mu}\\
    &\phantom{==}+c\int_{[\theta>0]}\big(\gamma^{s-4}\theta^{r-4}K^4+\gamma^s\|\theta^{r/4}A\|_{\infty,[\gamma>0]}^4\big)|\nabla^2A|^2\dd{\mu}+c\int_\Sigma\gamma^s\theta^r|\nabla A|^4|A|^2\dd{\mu}\\
    &\phantom{==}+c\,K^8\|A\|_{2,[\gamma\theta>0]}^2\tag{Lemma \ref{sobolev42higher}}
\\
    &\leq\varepsilon\int_\Sigma\gamma^s\theta^r|\nabla^4A|^2\dd{\mu}
    +c\int_\Sigma\big(\gamma^{s-2}\theta^{r-2}K^2+\gamma^s\theta^{r/2}\|\theta^{r/4}A\|_{\infty,[\gamma>0]}^2\big)|\nabla^3A|^2\dd{\mu}\\
    &\phantom{==}+c\int_{[\theta>0]}\big(\gamma^{s-4}\theta^{r-4}K^4+\gamma^s\|\theta^{r/4}A\|_{\infty,[\gamma>0]}^4\big)|\nabla^2A|^2\dd{\mu}\\
    &\phantom{==}+c\,\big(K^8+K^5\|\theta^{r/4}A\|_{\infty,[\gamma>0]}^3\big)\|A\|_{2,[\gamma\theta>0]}^2\tag{Proposition \ref{interpolation2higher}}
\\
    &\leq2\varepsilon\int_\Sigma\gamma^s\theta^r|\nabla^4A|^2\dd{\mu}+c\,\|\theta^{r/4}A\|_{\infty,[\gamma>0]}^4\int_{[\theta>0]}\gamma^s|\nabla^2A|^2\dd{\mu}\\
    &\phantom{==}+c\,\big(K^8+K^4\|\theta^{r/4}A\|_{\infty,[\gamma>0]}^4\big)\|A\|_{2,[\gamma\theta>0]}^2.\tag{Proposition \ref{interpolation1.1}}
\end{align*}
Finally, apply this estimate with sufficiently small $\varepsilon$ and $k=2$ in Proposition \ref{evIneq3} so that
\begin{align*}
    &\dv{t}\int_\Sigma\gamma^s|\nabla^2A|^2\dd{\mu}+\frac12\int_\Sigma\gamma^s\theta^r|\nabla^4A|^2\dd{\mu}\\
    &\leq c\,\|\theta^{r/4}A\|_{\infty,[\gamma>0]}^4\int_{[\theta>0]}\gamma^s|\nabla^2A|^2\dd{\mu}+c\,\big(K^8+K^4\|\theta^{r/4}A\|_{\infty,[\gamma>0]}^4\big)\|A\|_{2,[\gamma\theta>0]}^2.
\end{align*}
\item
For $m\geq3$,
\begin{align*}
    &\int_\Sigma\gamma^s\nabla^mA\ast\big[\nabla^m\big(\theta^r(P_3^2+P_5^0)\big)+\nabla^{m+1}\big(\theta^{r-1}\nabla\theta\ast(P_1^2+P_3^0)\big)\big]\dd{\mu}
\\
    &\leq\int_\Sigma\nabla^2(\gamma^2\nabla^mA)\ast\big[\nabla^{m-2}\big(\theta^r(P_3^2+P_5^0)\big)+\nabla^{m-1}\big(\theta^{r-1}\nabla\theta\ast(P_1^2+P_3^0)\big)\big]\dd{\mu}
\\
    &\leq c\int_\Sigma\big[\gamma^s|\nabla^{m+2}A|+\gamma^{s-1}K|\nabla^{m+1}A|+(\gamma^{s-2}K^2+\gamma^{s-1}K|A|)|\nabla^mA|\big]\\
    &\phantom{===}\cdot\big[\theta^r\big(|\nabla^mA|\,|A|^2+|\nabla^{m-1}A|\,|\nabla A|\,|A|\big)+\theta^{r-1}K\big(|\nabla^{m-1}A|\,|A|^2+|\nabla^{m+1}A|\\
    &\phantom{====}+|\nabla^mA|\,|A|+|\nabla^{m-1}A|\,|\nabla A|\big)+\theta^{r-2}K^2\big(|\nabla^mA|+|\nabla^{m-1}A|\,|A|\big)\\
    &\phantom{====}+\theta^{r-3}K^3|\nabla^{m-1}A|+|T|\big]\dd{\mu},
\end{align*}
where $T$ is a tensor that is supported on $[\gamma\theta>0]$ and can be described as a polynomial defined by operators $+$ and $\ast$, with variables being $A,\ldots,\nabla^{m-2}A$, with coefficients bounded by some $c(n,s,r,m,\varepsilon,K)$, at most of degree $(m+2)$, and without constant terms.
In particular, using H\"older's inequality, $\|T\|_2^2$ is bounded above by a quantity in the same form as how $\beta_{m,\gamma}$ is described.
Next, using Cauchy-Schwartz inequality and Proposition \ref{interpolation1.1}, we have
\begin{align*}
    &\int_\Sigma\gamma^s\nabla^mA\ast\big[\nabla^m\big(\theta^r(P_3^2+P_5^0)\big)+\nabla^{m+1}\big(\theta^{r-1}\nabla\theta\ast(P_1^2+P_3^0)\big)\big]\dd{\mu}
\\
    &\leq\varepsilon\int_\Sigma\gamma^s\theta^r|\nabla^{m+2}A|^2\dd{\mu}+c\int_\Sigma\gamma^{s-2}\theta^{r-2}K^2|\nabla^{m+1}A|^2\dd{\mu}\\
    &\phantom{==}+c\int_\Sigma\big(\gamma^{s-4}\theta^{r-4}K^4+\gamma^s\theta^r\|A\|_{\infty,[\gamma\theta>0]}^4\big)|\nabla^mA|^2\dd{\mu}\\
    &\phantom{==}+c\int_\Sigma\big(\gamma^{s-6}\theta^{r-6}K^6+\gamma^{s-2}\theta^{r-2}K^2\|A\|_{\infty,[\gamma\theta>0]}^4+\gamma^{s-2}\theta^{r-2}K^2\|\nabla A\|_{\infty,[\gamma\theta>0]}^2\\
    &\phantom{===}+\gamma^s\theta^r\|\nabla A\|_{\infty,[\gamma\theta>0]}^2\|A\|_{\infty,[\gamma\theta>0]}^2\big)|\nabla^{m-1}A|^2\dd{\mu}+c\,\|T\|_2^2
\\
    &\leq2\varepsilon\int_\Sigma\gamma^s\theta^r|\nabla^{m+2}A|^2\dd{\mu}\\
    &\phantom{==}+c\,\big(K^8 +\|A\|_{\infty,[\gamma\theta>0]}^8+\|\nabla A\|_{\infty,[\gamma\theta>0]}^4\big)\int_{[\gamma\theta>0]}|\nabla^{m-2}A|^2\dd{\mu}+c\,\|T\|_2^2.
\end{align*}
Finally, apply this estimate with sufficiently small $\varepsilon$ and $k=2$ in Proposition \ref{evIneq3} so that
\begin{align*}
    &\dv{t}\int_\Sigma\gamma^s|\nabla^mA|^2\dd{\mu}+\frac12\int_\Sigma\gamma^s\theta^r|\nabla^{m+2}A|^2\dd{\mu}
    \leq\beta_{m,\gamma}
\end{align*}
for some appropriate choice for $\beta_{m,\gamma}$.
\end{enumerate}
\end{proof}
Having $s=6$ and $r=20$ fixed, we will from now on omit them when describing dependence.
\section{A priori estimates and existence time}
Using Sobolev inequalities in section 3 and Gronwall's lemma, we derive a priori estimates for $L^2$ norms and $L^\infty$ norms.
As a result, we can derive the short-time existence result for Willmore flow.
\begin{CONV}
For $j=1,2,3$, let $\widetilde\gamma_j=\sigma_j\circ\widetilde\gamma$ and $\gamma_j=\sigma_j\circ\gamma=\widetilde\gamma_j\big|_\Sigma$, where each $\sigma_j$ is a function on $\bR$ such that
\begin{align*}\begin{cases}
    \sigma_j\text{ is increasing and smooth,}\\
    \sigma_j(x)=0\text{ for all }x\leq\frac{3-j}3\text{,}\\
    0<\sigma_j(x)<1\text{ for all }\frac{3-j}3<x<\frac{4-j}3,\\
    \sigma_j(x)=1\text{ for all }x\geq\frac{4-j}3\text{, and }\\
    |D\sigma_j(x)|\leq c\text{ and }|D^2\sigma_j(x)|\leq c\text{ for some universal constant }c.
\end{cases}\end{align*}
In particular, by section \ref{conventions}, $|D\widetilde\gamma_j|\leq c\,K$ and $|D^2\widetilde\gamma_j|\leq c\,K^2$ with some universal constant $c$.
\end{CONV}
\begin{LEM}\label{bootstrap1}
Let $m\geq3$ and $\beta_{m,\gamma_3}$ be as described in Proposition \ref{evIneq4} but with $\gamma$ replaced by $\gamma_3$.
Assuming (\ref{energyNonconcentrating}), we have
\begin{align*}
    \beta_{m,\gamma_3}\leq c,
\end{align*}
where $c=c(n,m,K,\alpha)$ and
\begin{align*}
    \alpha:=\sum_{k=0}^{m-1}\|\nabla^kA\|_{2,[\gamma>0]}.
\end{align*}
\end{LEM}
\begin{proof}
Throughout this proof, we let $c=c(n,m,K)$.
First, by Lemma \ref{sobolev43}, we have
\begin{align*}
    &\|A\|_{\infty,[\gamma_3>0]}
    \leq\|A\|_{\infty,[\gamma_2>0]}
    \leq c\,\|A\|_{2,[\gamma>0]}^{1/2}\big(\|\nabla^2A\|_{2,[\gamma>0]}^{1/2}+\|A\|_{2,[\gamma>0]}^{1/2}\big)
    \leq c\,\alpha.
\end{align*}
Next, also by Lemma \ref{sobolev43},
\begin{align*}
    \|\nabla A\|_{\infty,[\gamma_3>0]}
    &\leq c\,\|\nabla A\|_{2,[\gamma>0]}^{1/2}\big(\|\nabla^3A\|_{2,[\gamma>0]}^{1/2}+\|\nabla A\|_{2,[\gamma>0]}^{1/2}+\|\,|A|^4|\nabla A|^2\|_{1,[\gamma_2>0]}^{1/4}\big)\\
    &\leq c\,\|\nabla A\|_{2,[\gamma>0]}^{1/2}\big(\|\nabla^3A\|_{2,[\gamma>0]}^{1/2}+\|\nabla A\|_{2,[\gamma>0]}^{1/2}+\|A\|_{\infty,[\gamma_2>0]}\|\nabla A\|_{2,[\gamma_2>0]}^{1/2}\big)\\
    &\leq c\,(\alpha+\alpha^2).
\end{align*}
Next, consider $\|\nabla^jA\|_{p,[\gamma_3>0]}$, where $0\leq 
j\leq(m-2)$ and $3\leq p\leq(2m+4)$ are integers.
By Lemma \ref{sobolev1}, we have
\begin{align*}
   \|\nabla^jA\|_{p,[\gamma_3>0]}
   &\leq c\,\big(\|\nabla^jA\|_{2,[\gamma_2>0]}+\|\nabla^{j+1}A\|_{2,[\gamma_2>0]}+\|\nabla^jA\|_{2,[\gamma_2>0]}+\|A\|_{\infty,[\gamma_2>0]}\|\nabla^jA\|_{2,[\gamma_2>0]}\big)\\
   &\leq c\,(\alpha+\alpha^2).
\end{align*}
By the definition of $\beta_{m,\gamma_3}$, we have the desired result.
\end{proof}
\begin{PROP}\label{bootstrap2}
For all $k\geq0$, define
\begin{align*}
    \alpha_0(k)=\sum_{j=0}^k\|\nabla^jA\|_{2,[\gamma>0]}\bigg|_{t=0}.
\end{align*}
Assuming (\ref{energyNonconcentratingST}) for some $0<t_0<T$, we have
\begin{align*}
    \sup_{0\leq t\leq t_0}\|\nabla^mA\|_{2,[\gamma=1]}\leq c\big(n,K,m,t_0,\alpha_0(m)\big).
\end{align*}
\end{PROP}
\begin{proof}
The case $m=0$ is proved by hypothesis.
We set $m>0$ and assume that we have for each $k=0,\ldots,(m-1)$,
\begin{align*}
    \|\nabla^kA\|_{2,[\gamma=1]}\leq c\big(n,m,K,t_0,\alpha_0(k)\big).
\end{align*}
First, we have
\begin{align*}
    &\int_0^t\|\theta^{r/4}A\|_{\infty,[\gamma_3>0]}^4\dd{t'}\\
    &\leq c(n,K)\int_0^t\|A\|_{2,[\gamma_2>0]}^2\big(\|\theta^{r/2}\nabla^2A\|_{2,[\gamma_2>0]}^2+\|A\|_{2,[\gamma_2>0]}^2\big)\dd{t'}\tag{Corollary \ref{sobolev3}}\\
    &\leq c(n,K)\,\varepsilon_0\int_0^t\int_{[\gamma_1=0]}\theta^r|\nabla^2A|^2\dd{\mu}\dd{t'}+\varepsilon_0\,c(n,K,t_0)\\
    &\leq c(n,K,t_0).\tag{Lemma \ref{bootstrapLower}}
\end{align*}
In particular,
\begin{align*}
    \beta_{2,\gamma_3}\leq c(n,K,t_0).
\end{align*}
By the hypothesis and Lemma \ref{bootstrap1}, we also have
\begin{align*}
    \beta_{1,\gamma_3}\leq c(n,K),
\end{align*}
and for all $m\geq3$,
\begin{align*}
    \beta_{m,\gamma_3}\leq c\big(n,m,K,t_0,\alpha_0(m-1)\big).
\end{align*}
In particular, for all $m\geq1$,
\begin{align*}
    \int_0^t\beta_{m,\gamma_3}\dd{t'}\leq c\big(n,m,K,t_0,\alpha_0(m-1)\big).
\end{align*}
Next, by replacing $\gamma$ with $\gamma_3$ in Proposition \ref{evIneq4}, we have
\begin{align*}
    &\int_\Sigma\gamma_3^s|\nabla^mA|^2\dd{\mu}\\
    &\leq\int_\Sigma\gamma_3^s|\nabla^mA|^2\dd{\mu}+\frac12\int_0^t\int_\Sigma\gamma_3^s\theta^r|\nabla^{m+2}A|^2\dd{\mu}\dd{t'}\\
    &\leq\left(\int_\Sigma\gamma_3^s|\nabla^mA|^2\dd{\mu}\right)\bigg|_{t=0}+c(n,m,K)\int_0^t\left(\|\theta^{r/4}A\|_{\infty,[\gamma_3>0]}^4\int_\Sigma\gamma_3^s|\nabla^mA|^2\dd{\mu}\right)\dd{t'}\\
    &\phantom{==}+c(n,m,K)\,\sup_{0\leq t<T}\|\nabla^{m-1}A\|_{2,[\gamma>0]}^2\int_0^t\big(1+\|\theta^{r/4}A\|_{\infty,[\gamma_3>0]}^4\big)\dd{t'}+\int_0^t\beta_{m,\gamma_3}\dd{t'}\\
    &\leq c(n,m,K)\int_0^t\left(\|\theta^{r/4}A\|_{\infty,[\gamma_3>0]}^4\int_\Sigma\gamma_3^s|\nabla^mA|^2\dd{\mu}\right)\dd{t'}+c\big(n,m,K,t_0,\alpha_0(m)\big).
\end{align*}
Therefore, by Gronwall's lemma, we have
\begin{align*}
    \|\nabla^mA\|_{2,[\gamma=1]}^2
    &\leq\int_\Sigma\gamma_3^s|\nabla^mA|^2\dd{\mu}\\
    &\leq c\big(n,m,K,t_0,\alpha_0(m)\big)\exp\left(c(n,m,K)\int_0^t\|\theta^rA\|_{\infty,[\gamma>0]}^4\dd{t'}\right)\\
    &\leq c\big(n,m,K,t_0,\alpha_0(m)\big).
\end{align*}
\end{proof}
\begin{CORO}\label{bootstrapcoro}
Under the same settings, we have that for all $m\geq0$,
\begin{align*}
    \sup_{0\leq t\leq t_0}\|\nabla^mA\|_{\infty,[\gamma=1]}\leq c\big(n,m,K,t_0,\alpha_0(m+2)\big).
\end{align*}
\end{CORO}
\begin{proof}
By Lemma \ref{sobolev43},
\begin{align*}
    \|\nabla^mA\|_{\infty,[\gamma=1]}
    &\leq c(n,m,K)\,\|\nabla^mA\|_{2,[\gamma_3>0]}^{1/2}\big(\|\nabla^{m+2}A\|_{2,[\gamma_3>0]}^{1/2}+(1+\|A\|_{\infty,[\gamma_3>0]})\|\nabla^mA\|_{2,[\gamma_3>0]}^{1/2}\big)\\
    &\leq c\big(n,m,K,t_0,\alpha_0(m+2)\big).
\end{align*}
\end{proof}
\begin{PROP}\label{existence1}
Let $\Sigma$ be closed, and let $0\leq\widehat\theta\leq1$ be a smooth function on $\bR^n$ such that
\begin{align*}
    K_2:=\sup|D\widehat\theta|<\infty\phantom{=}\text{ and }\phantom{=}\sup|D^k\widehat\theta|\leq c(k)K_2^k,\phantom{=}\forall k\geq1.
\end{align*}
Then there exist $a_n>0$ and $c_0>0$, both depending only on $n$, such that whenever $f_0:\Sigma\to\bR^n$ satisfies
\begin{align*}
    \int_{\Sigma\cap B_\varrho(x)}|A|^2\dd{\mu}\leq e_0<\frac{\varepsilon_0}{a_n},\phantom{=}\forall x\in\bR^n
\end{align*}
for some $\varrho>0$, we can find a solution $f:\Sigma\times[0,T)\to\bR^n$ to equation (\ref{pdeCuttoff}) such that
\begin{align*}
    T\geq c_0^{-1}K^{-4}\left(\frac{\varepsilon_0}{a_n}-e_0\right),
\end{align*}
where $K=\max\{2/\varrho,K_2\}$ and $T$ is the maximum existence time.
\end{PROP}
\begin{proof}
Let $a_n$ be the number of balls of radius 1 in $\bR^n$ required to cover a ball of radius 2. Note that without loss of generality, we can assign these balls of radius 1 to have their centers in the ball of radius 2.
Assuming that $f$ exists with $T>0$, define
\begin{align*}
    e(t)=\sup_{x\in\bR^n}\int_{\Sigma\cap B_\varrho(x)}|A|^2\dd{\mu}\phantom{=}\text{at time }t.
\end{align*}
By hypothesis, $e(0)\leq e_0<\varepsilon_0$ and $e(t)$ is continuous.
We define
\begin{align*}
    t_0=\max\{0\leq t\leq T\,:\,\forall0\leq\tau<t\text{, }e(\tau)\leq\varepsilon_0\},
\end{align*}
which is always a positive number.
Moreover, we either have $t_0=T$ or $e(t_0)=\varepsilon_0$.

For each $x\in\bR^n$, we can find $\widehat\gamma$ such that
\begin{align*}
    \chi_{B_{\varrho/2}(x)}\leq\widehat\gamma\leq\chi_{B_\varrho(x)}
\end{align*}
as in Lemma \ref{cutoffConstruct} with $K_1=2/\varrho$, so that $K=\max\{K_1,K_2\}$.
By Corollary \ref{bootstrapcoro},
\begin{align*}
    \sup_{\substack{0\leq t\leq t_0\\ 0\leq t<T}}\|\nabla^mA\|_\infty
    =\sup_{x\in\bR^n}\sup_{\substack{0\leq t\leq t_0\\ 0\leq t<T}}\|\nabla^mA\|_{\infty,\Sigma\cap B_{\varrho/2}(x)}\leq c(n,m,K,t_0,f_0).
\end{align*}
In addition, as shown in the proof of \cite[Theorem 1.2]{Kuwert2002GradientFF}, we can show that for all $0\leq t<t_0$,
\begin{align}
    |\partial_x^mf(x,t)|,|\partial_x^m\partial_tf(x,t)|
    \leq c\big(n,m,K,t_0,f_0\big).
    \label{estimateBoundedness}
\end{align}
\begin{enumerate}[label=(\roman*)]
    \item Assume $\widehat\theta>0$ and $t_0=T$.\\
Since $\widehat\theta>0$, $f$ exists with $T>0$.
By the aforementioned estimate, $f(x,t)$ converges to a smooth function $f(x,T)=f_T(x)$ as $t\to T$.
Therefore, by short time existence for (\ref{pdeCuttoff}), we can extend the PDE to a longer time, a contradiction.
    \item Assume $\widehat\theta>0$ and $e(t_0)=\varepsilon_0$.\\
Since $\widehat\theta>0$, $f$ exists with $T>0$.
For all $0\leq t<t_0$, by Lemma \ref{bootstrapLower}, we have
\begin{align*}
    \int_{\Sigma\cap B_{\varrho/2}(x)}|A|^2\dd{\mu}\leq e_0+c\,K^4\varepsilon_0t=e_0+c_0K^4t,
\end{align*}
where $c=c(n)$ and we define
\begin{align*}
    c_0=c\,\varepsilon_0,
\end{align*}
which also depends only on $n$.
Since we can obtain
\begin{align*}
    \varepsilon_0=e(t_0)
    &=\sup_{x\in\bR^n}\int_{\Sigma\cap B_\varrho(x)}|A|^2\dd{\mu}\bigg|_{t=t_0}\\
    &\leq a_n\sup_{y\in\bR^n}\int_{\Sigma\cap B_{\varrho/2}(y)}|A|^2\dd{\mu}\bigg|_{t=t_0}\\
    &\leq a_n(e_0+c_0K^4t_0),
\end{align*}
we have
\begin{align*}
    t_0\geq c_0^{-1}K^{-4}\left(\frac{\varepsilon_0}{a_n}-e_0\right).
\end{align*}
    \item General case.\\
Let $\eta>0$ and replace $\widehat\theta$ with $\big(\eta+(1-\eta)\widehat\theta\big)$.
Since (i) cannot hold, by applying case (ii), we can find
\begin{align*}
    \widehat{f}:\Sigma\times\left[0,c_0^{-1}K^{-4}\left(\frac{\varepsilon_0}{a_n}-e_0\right)\right]\to\bR^n
\end{align*}
such that
\begin{align*}\begin{cases}
    \partial_t\widehat{f}=-\left(\eta+(1-\eta)\widehat\theta\circ \widehat{f}\right)^r\mathbf{W}(\widehat{f}),\\
    \widehat{f}\big|_{t=0}=f_0.
\end{cases}
\end{align*}
Moreover, we have (\ref{estimateBoundedness}) for all $\eta$ and $t$ without dependence on $\eta$ on the right hand side.
Therefore, as $\eta\tz$, there exists a subsequential limit $\widehat{f}\to f$ that is defined for all
\begin{align*}
    0\leq t\leq c_0^{-1}K^{-4}\left(\frac{\varepsilon_0}{a_n}-e_0\right)
\end{align*}
and solves (\ref{pdeCuttoff}).
\end{enumerate}
\end{proof}
\begin{THM}[Short time existence and minimal existence time]\label{existence3}
Let $f_0:\Sigma\to\bR^n$ be a smooth, complete, properly immersed surface in $\bR^n$.
Then there exist $\varepsilon_1>0$ and $c_1>0$, both depending only on $n$, such that whenever the initial energy concentration condition
\begin{align*}
    \int_{\Sigma\cap B_\varrho(x)}|A|^2\dd{\mu}\leq\varepsilon_1\phantom{=}\text{when }t=0\text{, }\forall x\in\bR^n
\end{align*}
holds for some $\varrho>0$, the Willmore flow equation (\ref{pdeOriginal}) has a smooth solution $f:\Sigma\times[0,T)\to\bR^n$.
Moreover, the maximal existence time $T$ is at least $c_1^{-1}\varrho^4$.
\end{THM}
\begin{proof}
Define
\begin{align*}
    \varepsilon_1=\frac{\varepsilon_0}{2a_n}
\end{align*}
and
\begin{align*}
    c_1=\frac{a_nc_0}{8\varepsilon_0}
    ,\phantom{=}\text{that is to say,}\phantom{=}
    c_0^{-1}\left(\frac2\varrho\right)^{-4}\frac{\varepsilon_0}{2a_n}=c_1^{-1}\varrho^4.
\end{align*}
Fix $K=K_2=2/\varrho$ and let
\begin{align*}
    \chi_{B_{R-\varrho/2}(0)}\leq\widehat\theta\leq\chi_{B_R(0)}
\end{align*}
as in Lemma \ref{cutoffConstruct}, where $R>\varrho/2$.

We claim that for all $R$, there exists a solution
\begin{align*}
    f_R:\Sigma\times[0,T_R)\to\bR^n
\end{align*}
that solves (\ref{pdeCuttoff}) and has the maximum existence time $T_R$ be at least $c_1^{-1}\varrho^4$.

Assume the contrary that either $f_R$ doesn't exist for any short time interval (in this case we will say $T_R=0$ for convenience) or it does but with $T_R<c_1^{-4}\varrho^4$.
In the latter case, we can obtain from Lemma \ref{bootstrapLower} that for all $0\leq t<T_R$,
\begin{align*}
    \int_{\Sigma\cap B_{\varrho/2}(x)}|A|^2\dd{\mu}
    \leq\varepsilon_1+c_0\left(\frac2\varrho\right)^4t
    =\frac{\varepsilon_0}{a_n}\cdot\frac{1+c_1\varrho^{-4}t}{2}<\frac{\varepsilon_0}{a_n}.
\end{align*}
Thus as in part (i) of the proof of Proposition \ref{existence1}, $f_R(x,t)$ converges to a smooth function $f(x,T_R)$ as $t\to T_R$, and we can extend $f_R$ to $\Sigma\times[0,T_R]$.
Next, whether $T_R$ is 0 or positive, we can extend the subset
\begin{align*}
    \{f_R(x,T_R)\,:\,x\in\Sigma\cap B_R(0)\}
\end{align*}
to a closed surface $S$.
By Proposition \ref{existence1}, we can find a solution $\widehat{f}_S$ to (\ref{pdeCuttoff}) with initial surface $S$.
Since $\theta>0$ only when $f_R(x,T_R)$ agree with $S$, we can extend $f_R$ to
\begin{align*}
    \widehat{f}_R(x,t)=\begin{cases}
    f_R(x,t)&\text{if }x\in\Sigma\text{ and }0\leq t\leq T_R,\\
    \widehat{f}_S\big(f_R(x,T_R),\,t-T_R\big)&\text{if }x\in\Sigma\cap B_R(0)\text{ and }T_R\leq t<T_R+\delta\text{, and}\\
    x&\text{if }x\in\Sigma\cmpl B_R(0)\text{ and }0\leq t<T_R+\delta,
    \end{cases}
\end{align*}
where the existence time $\delta>0$ for $\widehat{f}_S$ depends on $S$.
It turns out that $\widehat{f}_R$ is another solution to (\ref{pdeCuttoff}) with a longer existence time, a contradiction.
Therefore, $f_R$ exists with $T_R\geq c_1^{-1}\varrho^4$.

Finally, (\ref{estimateBoundedness}) holds for all $R$, $x\in\Sigma$, and $0\leq t\leq c_1^{-1}\varrho^4$, with $t_0$ replaced by $c_1^{-1}\varrho^4$.
Note that the right hand side doesn't depend on $R$.
Therefore, as $R\ti$, there exists a subsequential limit $f_R\to f$ such that each derivative converges locally uniformly, so that $f$ solves (\ref{pdeOriginal}) and is defined for $0\leq t\leq c_1^{-1}\varrho^4$.
\end{proof}
\begin{CORO}[Energy inequality]\label{energyIneq}
If $\mathcal{W}(f_0)<\infty$ and $f$ is the Willmore flow constructed in the theorem, then we have
\begin{align*}
    \int_\Sigma|A|^2\dd{\mu}+\int_0^t\int_\Sigma|\mathbf{W}(f)|^2\dd{\mu}\dd{t'}\leq\int_\Sigma|A|^2\dd{\mu}\bigg|_{t=0}.
\end{align*}
\end{CORO}
\begin{proof}
Along $f_R$, by the definition of variational derivative, we have
\begin{align*}
    \int_{f_R(\Sigma)}|A|^2\dd{\mu}+\int_0^t\int_{f_R(\Sigma)}\theta^r|\mathbf{W}(f_R)|^2\dd{\mu}\dd{t'}=\int_\Sigma|A|^2\dd{\mu}\bigg|_{t=0}
\end{align*}
since $\theta$ has compact support.
As $R\ti$, both integrands on the left hand side converge pointwise to the corresponding integrands for $f$.
Thus by Fatou's lemma,
\begin{align*}
    \int_{f(\Sigma)}|A|^2\dd{\mu}+\int_0^t\int_{f(\Sigma)}|\mathbf{W}(f)|^2\dd{\mu}\dd{t'}
    \leq\int_\Sigma|A|^2\dd{\mu}\bigg|_{t=0}.
\end{align*}
\end{proof}
\begin{CORO}\label{existenceCoro}
If $f_0$ satisfies $\mathcal{W}(f_0)<\infty$,
then there exists $f$ with $T>0$.
Moreover, if $\mathcal{W}(f_0)\leq a_n\varepsilon_1=\frac12\varepsilon_0$, then there exists $f$ with $T=\infty$.
\end{CORO}
\begin{proof}
\begin{enumerate}[label=(\roman*)]
    \item 
For the former case, take $R$ sufficiently large so that
\begin{align*}
    \int_{\Sigma_0\cmpl B_R(0)}|A|^2\dd{\mu}<\varepsilon_1.
\end{align*}
Since $f_0$ is proper, we can find a finite open cover $\{B_{r_k}(x_k)\}_{k=1}^N$ of $\closure{B}_{R+1}(0)$ so that for all $k$,
\begin{align*}
   \int_{\Sigma_0\cap B_{2r_k}(x_k)}|A|^2\dd{\mu}<\varepsilon_1.
\end{align*}
Let $\varrho=\min\{1,r_1,\ldots,r_N\}$.
As a result, for all $x\in\bR^n$, either $x\in\closure{B}_{R+1}(0)$ so that for some $k=k(x)$ we have
\begin{align*}
    \int_{\Sigma_0\cap B_\varrho(x)}|A|^2\dd{\mu}
    \leq\int_{\Sigma_0\cap B_{\varrho+r_k}(x_k)}|A|^2\dd{\mu}
    <\varepsilon_1,
\end{align*}
or $x\notin\closure{B}_{R+1}(0)$ so that
\begin{align*}
    \int_{\Sigma_0\cap B_\varrho(x)}|A|^2\dd{\mu}
    \leq\int_{\Sigma_0\cmpl B_R(0)}|A|^2\dd{\mu}
    <\varepsilon_1.
\end{align*}
We hence have $T\geq c_1^{-1}\varrho^4>0$ by Theorem \ref{existence3}.
    \item
For the latter case, we observe that along $f_R$,
\begin{align*}
    \int_{\Sigma\cap B_R(0)}|A|^2\dd{\mu}=\int_{\Sigma\cap B_R(0)}|A|^2\dd{\mu}\bigg|_{t=0}-\int_0^t\int_\Sigma\theta^r|\mathbf{W}(f)|^2\dd{\mu}\dd{t'}\leq\varepsilon_0
\end{align*}
for all $0\leq t<T_R$.
Corollary \ref{bootstrapcoro} hence applies and we have (\ref{estimateBoundedness}) for all $0\leq t\leq t_0=T_R$, provided $T_R<\infty$.
However, $f_R(x,t)$ converges as $t\to T_R$, a contradiction against $T_R<\infty$.
As a result of $T_R=\infty$, we can take a subsequential limit $f_R\to f$ with the functions being defined on $\Sigma\times[0,\infty)$.
\end{enumerate}
\end{proof}
\section{Uniqueness}
In this section, we consider Willmore flow $f:\Sigma\times[0,T)\to\bR^n$, where $T>0$ is not necessarily the maximal existence time and
\begin{align}
    f(\Phi(x;t),t)=f_0(x)+\eta(x,t),\label{normalExpression}
\end{align}
where $\eta$ is perpendicular to $T_x\Sigma_0$, i.e., $\eta$ is a 1-parameter family of sections of $N\Sigma_0$, the normal bundle of $\Sigma_0$, and for each $t$, $\Phi$ is an automorphism of $\Sigma$.
We also let $f_0(x)=f(x,0)$ denote the initial surface and assume $\eta\big|_{t=0}=0$ and $\Phi(x;0)=x$ as the initial condition.
In particular, we should solve
\begin{align*}
    \partial_tf\big|_{\Phi(x)}=-Df\big|_{\Phi(x)}\big(\partial_t\Phi\big|_x\big)+\partial_t\eta\big|_x.
\end{align*}
Note that the right hand side is uniquely determined by the other side as long as $T_{\Phi(x)}\Sigma\oplus N_x\Sigma_0=\bR^n$.
\begin{LEM}\label{gaugeBasic}
    Let $f:\Sigma\times[0,T)\to\bR^n$ be a family of surfaces.
    If
    \begin{align*}
        M=\sup_{x,t}\max_{\substack{v\in T_x\Sigma\\ |v|_g=1}}|\partial_tD_vf|<\infty,
    \end{align*}
    then there exists $t_1>0$, only depending on $M$, such that expression (\ref{normalExpression}) can be determined for all $0\leq t<\min(t_1,T)$.
\end{LEM}
\begin{proof}
    Consider any unit tangent vector $u_0\in T_x\Sigma_0$.
    Let
    \begin{align*}
        u=f_\ast(u_0)=v+w,\phantom{=}\text{where }v\in T_x\Sigma_0\text{ and }w\in N_x\Sigma_0.
    \end{align*}
    In particular, $|v|=1$ and $|w|=0$ when $t=0$.
    By hypothesis, we have
    \begin{align*}
        \sqrt{|\partial_tv|^2+|\partial_tw|^2}=|\partial_tu|\leq M|u|.
    \end{align*}
    In particular,
    \begin{align*}
        |v|\geq2-e^{Mt}>0\phantom{=}\text{when }t<\frac1M\log2.
    \end{align*}
    Therefore, $T_x\Sigma\oplus N_x\Sigma_0=\bR^n$, so that we can define $\pi$, the projection map from $\bR^n$ onto $T_x\Sigma$ along $N_x\Sigma_0$.
    As a result, we can determine
    \begin{align*}\begin{cases}
        \partial_t\Phi=-\pi(\partial_tf),\\
        \partial_t\eta=(I-\pi)(\partial_tf).
    \end{cases}\end{align*}
    Moreover, for all $u_0$,
    \begin{align*}\begin{cases}
        &|D_{u_0}\Phi|=|u_0|=1,\\
        &|\partial_tD_u\Phi|\leq\|\pi\|_\ast|\partial_tD_{\Phi_\ast u}f|\leq Me^{Mt}\,\|D\Phi\|_\ast,
    \end{cases}\end{align*}
    so that in particular,
    \begin{align*}
        2-e^{e^{Mt}-1}\leq|D_u\Phi|\leq e^{e^{Mt}-1}.
    \end{align*}
    In summary, $\eta$ and $\Phi$ in (\ref{normalExpression}) are well-determined within the interval $0\leq t<\min(t_1,T)$, where
    \begin{align*}
        t_1=\frac1M\log(1+\log2)<\frac1M\log2.
    \end{align*}
\end{proof}

If we assume further that $f(\Phi(x),t)$ solves the Willmore flow equation (\ref{pdeOriginal}), then $\eta$ solves the equation
\begin{align}\begin{cases}
    \partial_t\eta=-\mathbf{W}_N(\eta),\\
    \eta\big|_{t=0}=0;
\end{cases}\label{pdeNormal}
\end{align}
where
\begin{align*}
    \mathbf{W}_N(\eta)\big|_x-\mathbf{W}(f)\big|_x=Df\big(\partial_t\Phi\big|_{\Phi^{-1}(x)}\big)\in T_x\Sigma.
\end{align*}

We will also need the following volume estimate.
\begin{LEM}\label{gauge1}
There exists $\varepsilon_2>0$ such that $f=f_0+\eta$ defines an immersed surface whenever
\begin{align}
    \big\|\,|\eta|_{g_0}|B|_{g_0}+|\nabla\eta|_{g_0}^2\big\|_\infty\leq\varepsilon_2,\label{regularityNormalGraph}
\end{align}
where $g_0$ and $B$ denote the metric and the second fundamental form of $\Sigma_0$, respectively.
In fact, we have
\begin{align*}
    |g-g_0|_{g_0}\leq b,
\end{align*}
where $0<b<\frac12$ only depends on $\varepsilon_2$, and furthermore,
\begin{align*}
    \frac{\det(g)}{\det(g_0)}\geq1-2b>0.
\end{align*}
\end{LEM}
\begin{proof}
We can obtain
\begin{align*}
    \partial_if
    &=\partial_if_0-(g_0)^{k\ell}\inner\eta{B_{ik}}\partial_\ell f_0+\nabla_i\eta
\end{align*}
and
\begin{align*}
    g_{ij}
    &=(g_0)_{ij}-2\inner\eta{B_{ij}}+(g_0)^{k\ell}\inner\eta{B_{ik}}\inner\eta{B_{j\ell}}+\inner{\nabla_i\eta}{\nabla_j\eta},
\end{align*}
so that in particular,
\begin{align*}
    |g-g_0|\leq2|\eta|\,|B|+|\eta|^2|B|^2+|\nabla\eta|^2.
\end{align*}
As a result, we can choose any
\begin{align*}
    0<\varepsilon_2<\frac{\sqrt2-1}2,
\end{align*}
so that
\begin{align*}
    |g-g_0|
    \leq b
    =2\big(\varepsilon_2+\varepsilon_2^2\big)<\frac12
\end{align*}
whenever inequality (\ref{regularityNormalGraph}) holds.
Moreover, we have
\begin{align*}
    \frac{\det(g)}{\det(g_0)}
    \geq(1-b)^2-b^2=1-2b>0,
\end{align*}
so that $f(\Sigma)$ is an immersed surface.
\end{proof}
\begin{CONV}\verynewline\begin{itemize}
    \item $\delta T=T-\big(T\big|_{t=0}\big)$ as a tensor on $\Sigma_0$, where all vectors are considered as $\bR^n$-valued so that we can pull-back from $\Sigma$ to $\Sigma_0$ via the map $f_0(x)+\eta(x,t)$.
    When expressing $\delta T$ in terms of other tensors on $\Sigma_0$, the subscript $0$ that denotes $t=0$ may be dropped.
    \item Given tensors $T_1,\ldots,T_k$ on $\Sigma_0$ and non-negative integers $\alpha_1,\ldots,\alpha_k$, the notation
    \begin{align*}
        \widehat{P}_0\big(\nabla^{\alpha_1}T_1,\ldots,\nabla^{\alpha_k}T_k\big)
    \end{align*}
    denotes a polynomial defined by operators $+$ and $\ast$, with coefficients bounded by some $c(n)$, and with variables being $\nabla^iT_j$, running through $1\leq j\leq k$ and $0\leq i\leq\alpha_j$.
    \item When $T_k$ is not $B$,
    \begin{align*}
        \widehat{P}_1\big(\nabla^{\alpha_1}T_1,\ldots,\nabla^{\alpha_k}T_k\big)
    \end{align*}
    denotes a polynomial that is of the form $\widehat{P}_0\big(\nabla^{\alpha_1}T_1,\ldots,\nabla^{\alpha_k}T_k\big)$ while also every term has at least degree 1.
    \item When $T_k$ is $B$,
    \begin{align*}
        \widehat{P}_1\big(\nabla^{\alpha_1}T_1,\ldots,\nabla^{\alpha_{k-1}}T_{k-1},\nabla^{\alpha_k}B\big)
    \end{align*}
    denotes a polynomial that is of the form $\widehat{P}_0\big(\nabla^{\alpha_1}T_1,\ldots,\nabla^{\alpha_k}B\big)$ while also every term has at least degree 1 in some $\nabla^iT_j$ with $1\leq j\leq k-1$.
\end{itemize}\end{CONV}
We will compute in normal coordinates for $\Sigma_0$.
As mentioned in the proof of Lemma \ref{gauge1}, we already know
\begin{align*}
    &\delta(g_{\bR^n})=0,\\
    &\delta(\partial f)=\eta\ast B+\nabla\eta=\widehat{P}_1(\nabla\eta,B)\text{, and}\\
    &\delta(g)=\eta\ast B+\eta\ast\eta\ast B\ast B+\nabla\eta\ast\nabla\eta=\widehat{P}_1(\nabla\eta,B).
\end{align*}
We should also compute
\begin{align*}
    &\nabla_k\delta(g_{ij})
    =\nabla_k\big(-2\inner\eta{B_{ij}}+g^{k\ell}\inner\eta{B_{ik}}\inner\eta{B_{j\ell}}+\inner{\nabla_i\eta}{\nabla_j\eta}\big)
    =\widehat{P}_1(\nabla^2\eta,\nabla B).
\end{align*}
As a result,
\begin{align*}
    \delta(\partial g)=\partial\delta(g)=\nabla\delta(g)=\widehat{P}_1(\nabla^2\eta,\nabla B),
\end{align*}
so that
\begin{align*}
    \delta(\Gamma)=\delta(\partial g)+\delta(g^{-1})\ast\delta(\partial g)=\widehat{P}_1\big(\delta(g^{-1}),\nabla^2\eta,\nabla B\big),
\end{align*}
and also,
\begin{align*}
    \nabla_k\delta(g^{ij})
    &=\partial_k\delta(g^{ij})
    =\delta(\partial_kg^{ij})
    =\delta(-g^{ip}g^{jq}\partial_kg_{pq})
    =\widehat{P}_1\big(\delta(g^{-1}),\nabla^2\eta,\nabla B\big).
\end{align*}
\begin{PROP}\label{gauge2}
For $\eta$ that satisfies condition (\ref{regularityNormalGraph}),
\begin{align*}
    \mathbf{W}_N(\eta)
    &=\mathbf{W}(f_0)+\Laplace^2\eta
    +\nabla^4\eta\ast\widehat{P}_1\big(\delta(g^{-1}),\nabla\eta\big)+\nabla^3\eta\ast\widehat{P}_0\big(\delta(g^{-1}),\nabla\eta,\nabla B\big)\\
    &\phantom{==}+\widehat{P}_1\big(\delta(g^{-1}),\nabla^2\eta,\nabla^3B\big).
\end{align*}
\end{PROP}
\begin{proof}
First, we obtain
\begin{align*}
    \partial_i\partial_jf
    &=\partial_i\big(\partial_jf_0+\nabla_j\eta-(g_0)^{k\ell}\inner\eta{B_{jk}}\partial_\ell f_0\big)\\
    &=\partial_i\partial_jf_0+\nabla^2_{ij}\eta-(g_0)^{k\ell}\inner{\nabla_j\eta}{B_{ik}}\partial_\ell f_0-(g_0)^{k\ell}\inner{\nabla_i\eta}{B_{jk}}\partial_\ell f_0\\
    &\phantom{==}-(g_0)^{k\ell}\inner\eta{\nabla_iB_{jk}}\partial_\ell f_0-(g_0)^{k\ell}\inner\eta{B_{jk}}B_{i\ell}\\
    &=\partial_i\partial_jf_0+\nabla^2_{ij}\eta+\widehat{P}_1(\nabla\eta,\nabla B).
\end{align*}
Since we can assume normal coordinates on $\Sigma_0$, we have $\partial_i\partial_jf_0=B_{ij}$ and hence
\begin{align*}
    \delta(A_{ij})
    &=\delta\big(\partial_i\partial_jf-g^{pq}\inner{\partial_i\partial_jf}{\partial_pf}\partial_qf\big)\\
    &=\delta(\partial_i\partial_jf)-\big(g^{pq}+\delta(g^{pq})\big)\inner{\nabla^2_{ij}\eta}{\nabla_p\eta}\big(\partial_qf+\delta(\partial_qf)\big)+\widehat{P}_1(\delta(g^{-1}),\nabla\eta,\nabla B)\\
    &=\nabla^2_{ij}\eta+\nabla^2\eta\ast\widehat{P}_1\big(\delta(g^{-1}),\nabla\eta\big)+\widehat{P}_1\big(\delta(g^{-1}),\nabla\eta,\nabla B\big).
\end{align*}
Taking derivatives, we can conclude that
\begin{align*}
    \delta(\Laplace H)
    =\Laplace^2\eta+\nabla^4\eta\ast\widehat{P}_1\big(\delta(g^{-1}),\nabla\eta\big)+\nabla^3\eta\ast\widehat{P}_0\big(\delta(g^{-1}),\nabla\eta,\nabla B\big)+\widehat{P}_1\big(\delta(g^{-1}),\nabla^2\eta,\nabla^3B\big),
\end{align*}
and as a consequence,
\begin{align*}
    \delta(\mathbf{W}_N\big)
    &=\Laplace^2\eta+\nabla^4\eta\ast\widehat{P}_1\big(\delta(g^{-1}),\nabla\eta\big)+\nabla^3\eta\ast\widehat{P}_0\big(\delta(g^{-1}),\nabla\eta,\nabla B\big)+\widehat{P}_1\big(\delta(g^{-1}),\nabla^2\eta,\nabla^3B\big).
\end{align*}
\end{proof}
\begin{DEF}
If $\eta_1$, $\eta_2$ are normal vector fields that satisfy condition (\ref{regularityNormalGraph}), we denote
\begin{align*}
    G^{ij}(\eta_1,\eta_2)=g^{ij}\big|_{\eta=\eta_2}-g^{ij}\big|_{\eta=\eta_1}=\delta(g^{-1})\big|_{\eta=\eta_2}-\delta(g^{-1})\big|_{\eta=\eta_1}.
\end{align*}
\end{DEF}
Next, by distributive law, we can derive the following.
\begin{CORO}\label{pdeNormalDifference}
If $\eta_1,\eta_2$ satisfy (\ref{regularityNormalGraph}) and are two solutions for equation (\ref{pdeNormal}), then we can consider $\widetilde\eta=\eta_2-\eta_1$, which is a normal vector field on $\Sigma_0$ that satisfies
\begin{align*}\begin{cases}
    \partial_t\widetilde\eta+\Laplace^2\widetilde\eta
    =\sum_{k=0}^4\nabla^k\widetilde\eta\ast Q_k+G(\eta_1,\eta_2)\ast S,\\
    \widetilde\eta\big|_{t=0}=0,
\end{cases}\end{align*}
where $Q_4$ is of the form
\begin{align*}
    \widehat{P}_1\big(\delta(g^{-1})|_{\eta=\eta_1},\delta(g^{-1})|_{\eta=\eta_2},\nabla\eta_1,\nabla\eta_2\big),
\end{align*}
$Q_3$ is of the form
\begin{align*}
    \widehat{P}_0\big(\delta(g^{-1})|_{\eta=\eta_1},\delta(g^{-1})|_{\eta=\eta_2},\nabla\eta_1,\nabla\eta_2,\nabla B\big),
\end{align*}
$Q_2$ is of the form
\begin{align*}
    \widehat{P}_1\big(\delta(g^{-1})|_{\eta=\eta_1},\delta(g^{-1})|_{\eta=\eta_2},\nabla^2\eta_1,\nabla^2\eta_2,\nabla^3B\big),
\end{align*}
and $Q_1$, $Q_0$, and $S$ are of the form
\begin{align*}
    \nabla\widehat{P}_1\big(\delta(g^{-1})|_{\eta=\eta_1},\delta(g^{-1})|_{\eta=\eta_2},\nabla^3\eta_1,\nabla^3\eta_2,\nabla B\big)+\widehat{P}_0\big(\delta(g^{-1})|_{\eta=\eta_1},\delta(g^{-1})|_{\eta=\eta_2},\nabla^3\eta_1,\nabla^3\eta_2,\nabla^3B\big).
\end{align*}
\end{CORO}
\begin{LEM}\label{gInvDifference}
For all $\eta_1,\eta_2$ satisfying (\ref{regularityNormalGraph}),
\begin{align*}
    |G(\eta_1,\eta_2)|\leq c\,\big(|\widetilde\eta|\,|B|+|\nabla\widetilde\eta|\big),
\end{align*}
where $\widetilde\eta=\eta_2-\eta_1$ and $c=c(\varepsilon_2)$.
Also,
\begin{align*}
    |\nabla G(\eta_1,\eta_2)|
    \leq c\,\big(|\widetilde\eta|+|\nabla\widetilde\eta|+|\nabla^2\widetilde\eta|\big),
\end{align*}
where
\begin{align*}
    c=c\big(\varepsilon_2,\|B\|_{C^1},\|\eta_1\|_{C^2},\|\eta_2\|_{C^2}\big).
\end{align*}
\end{LEM}
\begin{proof}
Letting $\eta=t\eta_2+(1-t)\eta_1$, we have
\begin{align*}
    \pdv{t} g_{ij}
    &=\widetilde\eta\ast B+\widetilde\eta\ast\eta\ast B\ast B+\nabla\widetilde\eta\ast\nabla\eta,
\end{align*}
and hence
\begin{align*}
    \pdv{t} g^{ij}
    &=\frac{(-1)^{i+j}}{\big(g_{11}g_{22}-g_{12}^2\big)^2}\left(\big(g_{11}g_{22}-g_{12}^2\big)\pdv{t}g_{\underline{ij}}-g_{\underline{ij}}g_{22}\pdv{t}g_{11}-g_{\underline{ij}}g_{11}\pdv{t}g_{22}+2g_{\underline{ij}}g_{12}\pdv{t}g_{12}\right),
\end{align*}
where $\underline{i}=3-i$, etc.
In particular,
\begin{align*}
    \left|\pdv{t}g^{ij}\right|
    &\leq c\,\big((1-2b)^{-1}+(1-2b)^{-2}\big)\big(1+b^2\big)\big(|\widetilde\eta|\,|B|+|\widetilde\eta|\,|\eta|\,|B|^2+|\nabla\widetilde\eta|\,|\nabla\eta|\big)\\
    &\leq\frac{c\,(1+\varepsilon_2)(1+b^2)}{1-2b}\big(|\widetilde\eta|\,|B|+|\nabla\widetilde\eta|\big).
\end{align*}
As a result, we have
\begin{align*}
    \big|G^{ij}(\eta_1,\eta_2)\big|
    &\leq\sup_{0\leq t\leq1}\left|\pdv{t} g^{ij}\right|
    \leq\frac{c\,(1+\varepsilon_2)(1+b^2)}{1-2b}\big(|\widetilde\eta|\,|B|+|\nabla\widetilde\eta|\big).
\end{align*}
Next, as mentioned earlier, we have
\begin{align*}
    \nabla\delta(g^{-1})=\widehat{P}_1\big(\delta(g^{-1}),\nabla^2\eta,\nabla B\big).
\end{align*}
Therefore,
\begin{align*}
    \nabla G(\eta_1,\eta_2)
    &=G(\eta_1,\eta_2)\ast\widehat{P}_0\big(\delta(g^{-1})|_{\eta=\eta_1},\delta(g^{-1})|_{\eta=\eta_2},\nabla^2\eta_1,\nabla^2\eta_2,\nabla B\big)\\
    &\phantom{==}+\widetilde\eta\ast\widehat{P}_0\big(\delta(g^{-1})|_{\eta=\eta_1},\delta(g^{-1})|_{\eta=\eta_2},\nabla^2\eta_1,\nabla^2\eta_2,\nabla B\big)\\
    &\phantom{==}+\nabla\widetilde\eta\ast\widehat{P}_0\big(\delta(g^{-1})|_{\eta=\eta_1},\delta(g^{-1})|_{\eta=\eta_2},\nabla^2\eta_1,\nabla^2\eta_2,\nabla B\big)\\
    &\phantom{==}+\nabla^2\widetilde\eta\ast\widehat{P}_0\big(\delta(g^{-1})|_{\eta=\eta_1},\delta(g^{-1})|_{\eta=\eta_2},\nabla^2\eta_1,\nabla^2\eta_2,\nabla B\big),
\end{align*}
and hence
\begin{align*}
    |\nabla G(\eta_1,\eta_2)|
    \leq c\,\big(|\widetilde\eta|+|\nabla\widetilde\eta|+|\nabla^2\widetilde\eta|\big),
\end{align*}
where
\begin{align*}
    c=c\big(\varepsilon_2,\|B\|_{C^1},\|\eta_1\|_{C^2},\|\eta_2\|_{C^2}\big).
\end{align*}
\end{proof}
\begin{PROP}\label{uniquenessProp}
Let $f_0:\Sigma\looparrowright\bR^n$ be a complete, proper, immersed surface with $\|B\|_{C^3}<\infty$ and
\begin{align*}
    \liminf_{R\ti}R^{-4}\mu_0\big(B_R(0)\big)=0.
\end{align*}
If $\eta_i:\Sigma\times[0,T)\to\bR^n$ solves the Willmore flow equation (\ref{pdeNormal}) while satisfying (\ref{regularityNormalGraph}) and
\begin{align*}
    \sup_{0\leq t<T}\|\eta_i\|_{C^3}<\infty,\phantom{=}\forall i=1,2,
\end{align*}
then $\eta_1=\eta_2$.
\end{PROP}
\begin{proof}
For any $R>0$, we can find
\begin{align*}
    \chi_{B_R(0)}\leq\widehat\gamma\leq\chi_{B_{2R}(0)}
\end{align*}
as in Lemma \ref{cutoffConstruct} with $K=R^{-1}$ and $\gamma$ to be the restriction of $\widehat\gamma$ on $\Sigma$.
As in Lemma $\ref{evIneq2}$ but with $\theta=1$ and $V=0$, we have
\begin{align}
    &\dv{t}\int_\Sigma\gamma^s|\widetilde\eta|^2\dd{\mu}+\frac78\int_\Sigma\gamma^s|\nabla^2\widetilde\eta|^2\dd{\mu}\notag\\
    &\leq c\int_\Sigma\gamma^s\widetilde\eta\ast\big(Y+|\nabla B|^2\widetilde\eta+|B|^4\widetilde\eta\big)\dd{\mu}+c\,R^{-4}\int_\Sigma\gamma^{s-4}|\widetilde\eta|^2\dd{\mu},\label{uniqEnergyEst}
\end{align}
where $c$ is universal and
\begin{align*}
    Y=\sum_{k=0}^4\nabla^k\widetilde\eta\ast Q_k+G(\eta_1,\eta_2)\ast S
\end{align*}
by Corollary \ref{pdeNormalDifference}.
Next, we have
\begin{align*}
    \int_\Sigma\gamma^s\widetilde\eta\ast Y\dd{\mu}
    &\leq\int_\Sigma\nabla^2\big(\gamma^s\widetilde\eta\ast Q_4\big)\ast\nabla^2\widetilde\eta\dd{\mu}+\int_\Sigma\nabla\big(\gamma^s\widetilde\eta\ast Q_3\big)\ast\nabla^2\widetilde\eta\dd{\mu}\\
    &\phantom{==}+\int_\Sigma\gamma^s\widetilde\eta\ast\nabla^2\widetilde\eta\ast Q_2\dd{\mu}+\int_\Sigma\big(\gamma^s\widetilde\eta\ast\nabla\widetilde\eta\big)\ast Q_1\dd{\mu}\\
    &\phantom{==}+\int_\Sigma\big(\gamma^s\widetilde\eta\ast\widetilde\eta\big)\ast Q_0\dd{\mu}+\int_\Sigma\big(\gamma^s\widetilde\eta\ast G(\eta_1,\eta_2)\big)\ast S\dd{\mu}.
\end{align*}
Thus, using integration by parts, Cauchy-Scwartz inequality, Lemma \ref{gInvDifference}, and Proposition \ref{interpolation3},
\begin{align*}
    \int_\Sigma\gamma^s\widetilde\eta\ast Y\dd{\mu}
    &\leq\left(\frac1{16}+\|Q_4\|_\infty\right)\int_\Sigma\gamma^s|\nabla^2\widetilde\eta|^2\dd{\mu}+c\int_\Sigma(\gamma^s+\gamma^{s-4}R^{-4})|\widetilde\eta|^2\dd{\mu},
\end{align*}
where
\begin{align*}
    c
    &=c\big(\varepsilon_2,\|B\|_{C^3},\|\eta_1\|_{C^3},\|\eta_2\|_{C^3}\big).
\end{align*}
Moreover, we can choose $\varepsilon_2$ to be sufficiently small so that $\|Q_4\|_\infty\leq\frac1{16}$, and hence in inequality (\ref{uniqEnergyEst}), we have
\begin{align*}
    \dv{t}\int_\Sigma\gamma^s|\widetilde\eta|^2\dd{\mu}+\frac34\int_\Sigma\gamma^s|\nabla^2\widetilde\eta|^2\dd{\mu}\leq c\int_\Sigma\gamma^s|\widetilde\eta|^2\dd{\mu}+c\,R^{-4}\int_\Sigma\gamma^{s-4}|\widetilde\eta|^2\dd{\mu}.
\end{align*}
In particular,
\begin{align*}
    \dv{t}\int_\Sigma\gamma^s|\widetilde\eta|^2\dd{\mu}
    \leq c\int_\Sigma\gamma^s|\widetilde\eta|^2\dd{\mu}
    +c\,R^{-4}\mu_0\big(B_{2R}(0)\big).
\end{align*}
Since $\widetilde\eta=0$ at $t=0$, by Gronwall's lemma,
\begin{align*}
    \int_\Sigma\gamma^s|\widetilde\eta|^2\dd{\mu}
    \leq c\big(e^{ct}-1\big)\,R^{-4}\mu_0\big(B_{2R}(0)\big).
\end{align*}
Fixing $t$, we take $R\ti$ to conclude
\begin{align*}
    \int_\Sigma|\widetilde\eta|^2\dd{\mu}=\lim_{R\ti}\int_\Sigma\gamma^s|\widetilde\eta|^2\dd{\mu}=0,
\end{align*}
and hence $\widetilde\eta=0$ for all time, i.e., $\eta_1=\eta_2$.
\end{proof}
\begin{THM}\label{uniquenessCoro}
Assume that $f_0:\Sigma\to\bR^n$ is a smooth, complete, properly immersed surface in $\bR^n$ such that
\begin{align*}
    \liminf_{R\ti}R^{-4}\mu_0\big(B_R(0)\big)=0\text{, and}
\end{align*}
for some $\varrho>0$ and $M>0$,
\begin{align*}\begin{cases}
    \displaystyle\int_{\Sigma_0\cap B_\varrho(x)}|B|^2\dd{\mu_0}\leq\varepsilon_1,&\forall x\in\bR^n\\
    \displaystyle\int_{\Sigma_0\cap B_\varrho(x)}|\nabla^kB|^2\dd{\mu_0}\leq M,&\forall x\in\bR^n\text{ and }k=1,\ldots,5,
\end{cases}\end{align*}
where $\varepsilon_1$ is as given in Theorem \ref{existence3}.
Let $f=f_1$ and $f=f_2$ be two solutions to the Willmore flow equation (\ref{pdeOriginal}), then there exists $t_3>0$, only depending on $n$, $\varrho$, and $M$, such that $f_1=f_2$ for all $0\leq t<\widehat{T}=\min(t_3,T)$.
\end{THM}
\begin{proof}
Either let $f$ denote $f_1$ or $f_2$.
As shown in the proof of Theorem \ref{existence3}, we have
\begin{align*}
    \int_{\Sigma\cap B_{\varrho/2}(x)}|A|^2\dd{\mu}\leq\varepsilon_0,\phantom{=}\forall x\in\bR^n\text{ and }0\leq t\leq\min(t_0,T),
\end{align*}
where $t_0$ only depends on $n$ and $\varrho$.
Therefore, we can apply Corollary \ref{bootstrapcoro} and obtain that
\begin{align*}
    \sup_{0\leq t<\min(t_0,T)}\|\nabla^kA\|_\infty\leq c(t_0,\varrho,M),\phantom{=}\forall k=0,\ldots,3.
\end{align*}
In particular, for all $0\leq t\leq\min(t_0,T)$,
\begin{align*}
    \sup_{x\in\Sigma}\max_{\substack{v\in T_x\Sigma\\ |v|=1}}|\partial_tD_vf|=\|\nabla\mathbf{W}(f)\|_\infty\leq\|\nabla^3A\|_\infty+c\,\|\nabla A\|_\infty\|A\|_\infty^2\leq c(n,\varrho,M).
\end{align*}
Thus by Lemma \ref{gaugeBasic}, there exists $t_1$ such that $f=f_0+\eta$ can be determined for all $0\leq t<\min(t_1,T)$, where $0<t_1\leq t_0$ only depends on $n$, $\varrho$, and $M$.

Next, we claim that
\begin{align*}\begin{cases}
    \sup_{0\leq t<t_1}\|\partial_t\nabla^k\eta\|_\infty<\infty&\forall k=0,1\text{, and}\\
    \sup_{0\leq t<t_1}\|\nabla^k\eta\|_\infty<\infty,&\forall k=2,3.
\end{cases}\end{align*}
If the claim holds, since $\eta=0$ at $t=0$, there exists $t_2$ such that (\ref{regularityNormalGraph}) holds for all $0\leq t<\min(t_2,T)$, where $0<t_2\leq t_1$.
Moreover, we can apply Proposition \ref{uniquenessProp} and conclude that $f_1=f_2$.

To prove the first part of the claim, we have
\begin{align*}
    \|\eta\|_\infty\leq\|I-\pi\|_\ast\|\mathbf{W}(f)\|_\infty\phantom{=}\text{ and }\phantom{=}
    \|\nabla\eta\|_\infty\leq\|I-\pi\|_\ast\|\nabla\mathbf{W}(f)\|_\infty,
\end{align*}
so that the required upper bound can be found, where $\pi$ is as defined in Lemma \ref{gaugeBasic}.
To prove the second part of the claim, recall that
\begin{align*}
    \delta(A_{ij})
    &=\nabla^2_{ij}\eta+\nabla^2\eta\ast\widehat{P}_1\big(\delta(g),\nabla\eta\big)+\widehat{P}_1\big(\delta(g),\nabla\eta,\nabla B\big)\\
    &=\nabla^2_{k\ell}\eta\big[\delta^k_i\delta^\ell_j+\widehat{P}_1\big(\delta(g),\nabla\eta\big)\big]+\widehat{P}_1\big(\delta(g),\nabla\eta,\nabla B\big).
\end{align*}
Therefore, if $0<t_3\leq t_2$ is chosen to be sufficiently small, then for all $0\leq t<\min(t_3,T)$,
\begin{align*}
    \delta^k_i\delta^\ell_j+\widehat{P}_1\big(\delta(g),\nabla\eta\big),
\end{align*}
as an $4\times4$ matrix with rows indexed by $(i,j)$ and columns indexed by $(k,\ell)$, is invertible with determinant at least $c^{-1}$, where $c=c(n,\varrho,M)>0$.
As a result,
\begin{align*}
    \|\nabla^2\eta\|_\infty\leq c.
\end{align*}
A similar discussion regarding $\delta(\nabla A)$ shows
\begin{align*}
    \|\nabla^3\eta\|_\infty\leq c.
\end{align*}
As mentioned, these conditions together prove the theorem.
\end{proof}
\section{Gap phenomena}
Assuming small total energy, we derive rigidity results for a Willmore surface to be a plane and a Willmore flow to converge to a plane when given sufficient conditions to converge to some surface.
\begin{PROP}\label{gap1}
If $f:\Sigma\times[0,T)\to\bR^n$ is a Willmore flow and $\mathcal{W}(f_0)<\frac12\varepsilon_0$, then
\begin{align*}
    \int_\Sigma|A|^2\dd{\mu}+\frac12\int_0^t\int_\Sigma\big(|\nabla A|^2+|A|^6\big)\dd{\mu}\dd{t'}\leq\int_\Sigma|A|^2\dd{\mu}\bigg|_{t=0}
\end{align*}
for all $t$.
\end{PROP}
\begin{proof}
For any $R>0$, we can find
\begin{align*}
    \chi_{B_R(0)}\leq\widehat\gamma\leq\chi_{B_{2R}(0)}
\end{align*}
as in Lemma \ref{cutoffConstruct} with $K=R^{-1}$ and $\gamma$ to be the restriction of $\widehat\gamma$ on $\Sigma$.
Consider
\begin{align*}
    e(t,R)=\int_{\Sigma\cap B_R(0)}|A|^2\dd{\mu}\phantom{=}\text{at time }t,
\end{align*}
and
\begin{align*}
    t_0(R)=\max\{t\in[0,T]\,:\,\forall\tau\in[0,t)\text{, }e(\tau,R)\leq\varepsilon_0\}.
\end{align*}
As in part (ii) of the proof of Theorem \ref{existence3}, a continuity argument using Lemma \ref{bootstrapLower} shows that
\begin{align*}
    t_0(R)\geq\min\left\{T\text{, }c^{-1}R^4\left(1-\varepsilon_0^{-1}\int_\Sigma|A|^2\dd{\mu}\bigg|_{t=0}\right)\right\}.
\end{align*}
However, by definition, $t_0(R)$ is decreasing in $R$, so we can take $R\ti$ on the right hand side and obtain $t_0(R)=T$ for all $R$.
Equivalently, for all $0\leq t<T$, $e(t,R)\leq\varepsilon_0$ and hence
\begin{align*}
    \int_\Sigma|A|^2\dd{\mu}=\lim_{R\ti}e(t,R)\leq\varepsilon_0.
\end{align*}
Next, we use Lemma \ref{bootstrapLower} on all $0\leq t<T$ and obtain by monotone convergence theorem that:
\begin{align*}
    &\int_\Sigma|A|^2\dd{\mu}+\frac12\int_0^t\int_\Sigma\big(|\nabla^2A+|A|^6\big)\dd{\mu}\dd{t'}\\
    &=\lim_{R\ti}\left(\int_{\Sigma\cap B_R(0)}|A|^2\dd{\mu}+\frac12\int_0^t\int_{\Sigma\cap B_R(0)}\big(|\nabla^2A+|A|^6\big)\dd{\mu}\dd{t'}\right)\\
    &\leq\liminf_{R\ti}\left(\int_{\Sigma\cap B_{2R}(0)}|A|^2\dd{\mu}\bigg|_{t=0}+c\,R^{-4}\varepsilon_0t\right)\\
    &=\int_\Sigma|A|^2\dd{\mu}\bigg|_{t=0}.
\end{align*}
\end{proof}
\begin{THM}[Gap rigidity]\label{gap2}
If $\Sigma_0$ is a complete, smooth, properly immersed Willmore surface with $\mathcal{W}(f_0)\leq\frac12\varepsilon_0$, then $\Sigma_0$ is a plane.
\end{THM}
\begin{proof}
By hypotheses, $f(x,t)=f_0(x)$ is a Willmore flow.
Since $\mathcal{W}(f)=\mathcal{W}(f_0)$, by Proposition \ref{gap1},
\begin{align*}
    \int_\Sigma|A|^6\dd{\mu}=0.
\end{align*}
That is, $A$ vanishes globally.
Therefore, $\Sigma_0$ is a plane.
\end{proof}
We will consider the following definition for convergence of surfaces, as described in \cite[Theorem 4.2]{10.4310/jdg/1090348128}.
\begin{DEF}\label{defCvg}
Let $f_j:\Sigma_j\to\bR^n$ and $\widehat{f}:\widehat\Sigma\to\bR^n$ be properly immersed surfaces without boundary.
We say that $\Sigma_j$ converges to $\widehat\Sigma$ (locally smoothly up to diffeomorphisms) if we can find numbers $r_i$ and functions $\varphi_i$, $u_i$ that satisfy the following.
\begin{itemize}
    \item $r_i$ is an increasing sequence of positive numbers such that $\lim_{i\ti}r_i=\infty$;
    \item For all $i$, $\varphi_i$ is a diffeomorphism
    \begin{align*}
        \varphi_i:\widehat{f}(\widehat\Sigma)\cap B_{r_i}(0)\,\simto\,U_i\subset\Sigma_i,
    \end{align*}
    and $u_i$ is a smooth normal vector field
    \begin{align*}
        u_i:\widehat{f}(\widehat\Sigma)\cap B_{r_i}(0)\to\bR^n,
    \end{align*}
    such that
    \begin{align*}
        f_i\circ\varphi_i=\widehat{f}+u_i;
    \end{align*}
    \item For all $R>0$, there exists $i_0=i_0(R)$ such that for all $i\geq i_0$,
    \begin{align*}
        \{p\in\Sigma_i:|f_i(p)<R\}\subset U_i;
    \end{align*}
    and
    \item For each $m\geq0$,
    \begin{align*}
        \lim_{i\ti}\|\nabla^mu_j\|_{\infty,\widehat{f}(\widehat\Sigma)\cap B_{r_i}(0)}=0.
    \end{align*}
\end{itemize}
\end{DEF}
\begin{CORO}[Low energy convergence]\label{gap3}
Let $f:\Sigma\times[0,\infty)\to\bR^n$ be a solution to (\ref{pdeOriginal}).
Assume that $\mathcal{W}(f_0)\leq\frac12\varepsilon_0$ and that $\mu\big(B_R(0)\big)\leq c(R)$ for all $t\in[0,\infty)$ and $R>0$.
Then as $t\ti$, any subsequence has a further subsequence such that $\Sigma$ converges to a plane $L:\bR^2\to\bR^n$ in the sense as in Definition \ref{defCvg}.
\end{CORO}
\begin{proof}
By Proposition \ref{gap1},
\begin{align*}
    \int_\Sigma|A|^2\dd{\mu}
    \leq\int_\Sigma|A|^2\dd{\mu}\bigg|_{t=0}\leq\varepsilon_0.
\end{align*}
Therefore, by taking $\varrho$ to be arbitrarily big in \cite[Theorem 3.5]{10.4310/jdg/1090348128}, we have
\begin{align*}
    \|\nabla^kA\|_\infty
    \leq c(k)\varepsilon_0t^{-\frac{k+1}4},\phantom{=}\text{for all }k\geq0.
\end{align*}
Together with the area bound, we can apply \cite[Theorem 4.2]{10.4310/jdg/1090348128} and find some properly immersed surface $L$ to be the limit of $\Sigma$.
In particular,
\begin{align*}
    \|A\|_{\infty,L}\leq\limsup_{t\ti}\|A\|_{\infty,\Sigma}=0,
\end{align*}
and hence $L$ is a plane.
\end{proof}
Alternatively, we also obtain the following convergence result:
\begin{CORO}\label{gap4}
Let $f:\Sigma\times[0,\infty)\to\bR^n$ be a solution to (\ref{pdeOriginal}).
Assume that $\mathcal{W}(f_0)\leq\frac12\varepsilon_0$ and that
\begin{align*}
    \liminf_{R\ti}R^{-2}\mu_0\big(B_R(0)\big)<\infty.
\end{align*}
Then as $t\ti$, any subsequence has a further subsequence such that $\Sigma$ converges to a plane $L:\bR^2\to\bR^n$ in the sense as in Definition \ref{defCvg}.
\end{CORO}
\begin{proof}
We denote $c=c(n)$.
For any $R>0$, we can find
\begin{align*}
    \chi_{B_R(0)}\leq\widehat\gamma\leq\chi_{B_{2R}(0)}
\end{align*}
as in Lemma \ref{cutoffConstruct} with $K=R^{-1}$ and $\gamma$ to be the restriction of $\widehat\gamma$ on $\Sigma$.
Along the Willmore flow, by \cite[Theorem 3.5]{10.4310/jdg/1090348128},
\begin{align*}
    \dv{t}\int_\Sigma\gamma^4\dd{\mu}
    &=\int_\Sigma\big(\inner{-\Laplace H-Q(A^0)H}{-\gamma^4H}+s\gamma^3\inner{-\Laplace H-Q(A^0)H}{D\widehat\gamma}\big)\dd{\mu}\\
    &\leq\int_\Sigma\gamma^4\big(-|\nabla H|^2+|A^0|^2|H|^2\big)\dd{\mu}\\
    &\phantom{=}+c\int_\Sigma\Big(|\nabla H|\,\big[\gamma^3|D\widehat\gamma|\,|A|+\gamma^2\big(|D\widehat\gamma|^2+|D^2\widehat\gamma|\big)\big]+\gamma^3|D\widehat\gamma|\,|A^0|^2|H|\Big)\dd{\mu}\\
    &\leq\int_\Sigma\gamma^4\big(-|\nabla H|^2+|A^0|^2|H|^2\big)\dd{\mu}+R^{-2}\int_\Sigma\gamma^4\dd{\mu}\\
    &\phantom{=}+c\,R^{-1}\big(\|A\|_{2,[\gamma>0]}^2+\|\nabla H\|_{2,[\gamma>0]}^2+\|H\|_{\infty,[\gamma>0]}\|A^0\|_{2,[\gamma>0]}^2\big)+c\,R^{-2}\|\nabla H\|_{2,[\gamma>0]}^2\\
    &\leq\int_\Sigma\gamma^4\big(-|\nabla H|^2+|A^0|^2|H|^2\big)\dd{\mu}+R^{-2}\int_\Sigma\gamma^4\dd{\mu}\\
    &\phantom{=}+c\,R^{-1}\big(\varepsilon_0+\varepsilon_0t^{-\frac12}+\varepsilon_0^{\frac32}t^{-\frac14}\big)+c\,R^{-2}\varepsilon_0t^{-\frac12}.
\end{align*}
Using \ref{SimonsA0} of Lemma \ref{eqn2.10}, Gauss-Codazzi equations, and \cite[Theorem 3.5]{10.4310/jdg/1090348128}, we have (cf. \cite[equation 68]{10.4310/jdg/1090348128}):
\begin{align*}
    &\int_\Sigma\gamma^4\big(-|\nabla H|^2+|A^0|^2|H|^2\big)\dd{\mu}\\
    &\leq-2\int_\Sigma\gamma^4|\nabla A^0|^2\dd{\mu}+c\,R^{-1}\int_\Sigma\gamma^3|A^0|\big(|\nabla H|+|\nabla A^0|\big)\dd{\mu}+c\int_\Sigma\gamma^4|A^0|^4\dd{\mu}\\
    &\leq c\,\|A^0\|_{\infty,[\gamma>0]}^4\int_\Sigma\gamma^4\dd{\mu}+c\,R^{-1}\big(\varepsilon_0+\varepsilon_0t^{-\frac12}\big).
\end{align*}
Therefore,
\begin{align*}
    \dv{t}\int_\Sigma\gamma^4\dd{\mu}
    \leq\big(R^{-2}+c\,\|A^0\|_{\infty,[\gamma>0]}^4\big)\int_\Sigma\gamma^4\dd{\mu}+c\,R^{-1}\big(\varepsilon_0+\varepsilon_0t^{-\frac12}+\varepsilon_0^{\frac32}t^{-\frac14}\big)+c\,R^{-2}\varepsilon_0t^{-\frac12}.
\end{align*}
In addition, since
\begin{align*}
    \int_{[\gamma>0]}|A^0|^2\dd{\mu}\bigg|_{t=0}\leq\int_\Sigma|A|^2\dd{\mu}\bigg|_{t=0}\leq\varepsilon_0,
\end{align*}
by \cite[Proposition 3.4]{10.4310/jdg/1090348128}, we have
\begin{align*}
    \int_0^t\|A^0\|_{\infty,[\gamma>0]}^4\dd{\tau}\leq c\,\varepsilon_0(1+R^{-4}t),
\end{align*}
and hence by Gronwall's lemma,
\begin{align*}
    \int_\Sigma\gamma^4\dd{\mu}
    \leq\left(\int_\Sigma\gamma^4\dd{\mu}\bigg|_{t=0}+c\,R^{-1}\big(\varepsilon_0t+\varepsilon_0t^{\frac12}+\varepsilon_0^{\frac32}t^{\frac34}\big)+c\,R^{-2}\varepsilon_0t^{\frac12}\right)e^{R^{-2}t+c\,\varepsilon_0(1+R^{-4}t)}.
\end{align*}
In particular,
\begin{align*}
    \mu\big(B_R(0)\big)\leq\left(\mu_0\big(B_{2R}(0)\big)+c\,R^{-1}\big(\varepsilon_0t+\varepsilon_0t^{\frac12}+\varepsilon_0^{\frac32}t^{\frac34}\big)+c\,R^{-2}\varepsilon_0t^{\frac12}\right)e^{R^{-2}t+c\,\varepsilon_0(1+R^{-4}t)}.
\end{align*}
Next, by monotonicity formula, for any $0<r<R$,
\begin{align*}
    r^{-2}\mu\big(B_r(0)\big)
    \leq c\,\left(R^{-2}\mu\big(B_R(0)\big)+\int_{\Sigma\cap B_R(0)}|H|^2\dd{\mu}\right).
\end{align*}
Moreover, fixing $r,t$ and letting $R\ti$,
\begin{align*}
    \mu\big(B_r(0)\big)
    \leq c\,\Big(\liminf_{R\ti}\big[R^{-2}\mu_0\big(B_{2R}(0)\big)\big]\,e^{c\,\varepsilon_0}+\varepsilon_0\Big)r^2,
\end{align*}
which is an area bound that is independent of the time variable $t$.
Therefore, the statement can be proved by Corollary \ref{gap3}.
\end{proof}

\bibliographystyle{alpha}
\bibliography{ref}
\end{document}